\documentclass[11pt]{amsart}  
\usepackage{amsthm, amsmath, amscd, amssymb, latexsym, stmaryrd, color}

\usepackage[all]{xypic}
\usepackage{mathtools}
\input xypic

\usepackage[top=3.0cm, bottom=3.0cm, left=3.0cm, right=3.0cm]{geometry}

\theoremstyle{plain}
\newtheorem{thm}{Theorem}
\newtheorem{prop}[thm]{Proposition}
\newtheorem{theorem}{Theorem}[section]
\newtheorem{lemma}[theorem]{Lemma}
\newtheorem{lemma-def}[theorem]{Lemma-Definition}
\newtheorem{proposition}[theorem]{Proposition}
\newtheorem{prop-def}[theorem]{Proposition-Definition}
\newtheorem{corollary}[theorem]{Corollary}
\newtheorem{conjecture}[theorem]{Conjecture}

\theoremstyle{definition}

\newtheorem{definition}[theorem]{Definition}

\newtheorem{notation}[theorem]{Notation}
\newtheorem{remark}[theorem]{Remark}

\theoremstyle{remark}
\newtheorem*{ack}{Acknowledgement}

\usepackage[mathscr]{eucal}
\usepackage{graphics, graphpap}
\usepackage{array, tabularx, longtable}
\usepackage{color}
\numberwithin{equation}{section}

\def\Var{\mathrm{Var}}

\def\sp{\mathrm{sp}}
\def\pr{\mathrm{pr}}

\def\sing{\mathrm{sing}}

\def\ord{\mathrm{ord}}

\def\Jac{\mathrm{Jac}}

\def\Spec{\mathrm{Spec}}
\def\Proj{\mathrm{Proj}}
\def\Spf{\mathrm{Spf}}
\def\Spm{\mathrm{Spm}}

\def\rig{\mathrm{rig}}
\def\Id{\mathrm{Id}}

\def\Gr{\mathrm{Gr}}
\def\sr{\mathrm{sr}}
\def\Hom{\mathrm{Hom}}

\def\MV{\mathrm{MV}}

\def\flat{\mathrm{flat}}
\def\red{\mathrm{red}}

\def\x{\mathbf{x}}

\def\L{\mathbb{L}}

\title{Equivariant motivic integration on special formal schemes}  
\author{L\^e Quy Thuong}
\date{}
\address{University of Science, Vietnam National University, Hanoi 
	\newline
	\indent 334 Nguyen Trai Street, Thanh Xuan District, Hanoi, Vietnam}
\email{leqthuong@gmail.com}

\author{Nguyen Hong Duc}
\address{$^{\dag}$TIMAS, Thang Long University, \newline \indent Nghiem Xuan Yem, Hanoi, Vietnam.} 
\email{duc.nh@thanglong.edu.vn}

\thanks{This research is funded by Vietnam National Foundation for Science and Technology Development (NAFOSTED) under grant number 101.04-2024.09.}


\begin{document}           

\begin{abstract}
In this article, we construct an equivariant version of motivic integration on special formal schemes that generalizes our previous work for algebraic varieties. Pointing out the existence of an equivariant N\'eron smoothening for a flat generically smooth special formal scheme, we prove a change of variables formula in this integration. Finally, the article introduces the motivic Milnor fiber of a formal power series. It predicts that this quantity is the right one to define the motivic Milnor fiber of a germ of complex analytic functions.


\end{abstract}

\maketitle  

\section{Introduction}\label{sec1}
In 1995, with the help of $p$-adic integration and the Weil conjectures (proved by Deligne), Batyrev \cite{Ba} obtained an important result in birational geometry and mathematical physics that birational Calabi-Yau varieties have the same Betti numbers. Immediately after this event, Kontsevich in his seminar talk at Orsay on December 7, 1995 introduced a new idea which approaches directly to Batyrev's Theorem without using $p$-adic integration. Kontsevich's method involves arc spaces and the Grothendieck ring of varieties, which brings about the birth of geometric motivic integration. Nowadays, this kind of integration becomes one of the common central objects of algebraic geometry, singularity theory, mathematical physics. From algebraic varieties to formal schemes, the development of geometric motivic integration is contributed crucially by Denef-Loeser \cite{DL1, DL2}, Sebag \cite{Se}, Loeser-Sebag \cite{LS}, Nicaise-Sebag \cite{NS1, NS, NS2, NS3}, Nicaise \cite{Ni2}, and many others. Another point of view on motivic integration known as arithmetic motivic integration was also strongly developed due to the approaches of Denef-Loeser over $p$-adic fields \cite{DL3}, Cluckers-Loeser \cite{CL05, CL08, CL10}, Hrushovski-Kazhdan \cite{HK} and Hrushovski-Loeser \cite{HL16} using model theory with respect to different languages. It is shown in \cite{CHL} that arithmetic motivic integration has an important application to the fundamental lemma. For another theory of motivic integration that specializes to both of arithmetic and geometric points of view, we can also refer to more recent works such as \cite{GY}, \cite{CGH}, \cite{CR}. 

It is natural to build an equivariant version of geometric motivic integration, which is very useful for applications to singularity theory. In fact, we can view the monodromy action on the classical Milnor fiber from natural actions of the group schemes of roots of unity on the contact loci of the singularity. The previous work \cite{LN} developed the equivariant motivic integration in the inheritance of Denef-Loeser's classical motivic integration for stable semi-algebraic subsets of arc spaces of algebraic varieties, in which we work with good actions of finite and profinite group schemes. In a formal setting, Hartmann \cite{Hartmann} recently has extended the theory of Sebag and Loeser in \cite{Se} and \cite{LS} to an equivariant version with respect to abelian finite groups (other than group schemes). 

Let $R$ be a complete discrete valuation ring whose fraction field $K$ and residue field $k$ have the same characteristic. Fix a uniformizing parameter $\varpi\in R$, denote $R_n=R/(\varpi^{n+1})$ and $R(n)=R[\tau]/(\tau^n-\varpi)$. The present work is a continuation of \cite{LN}, which aims to reach a general theory of equivariant geometric motivic integration for special formal schemes endowed with good adic actions of finite and profinite group $k$-schemes. Important tools for our construction are Theorems \ref{fund2} and \ref{keycoro}, which extend \cite[Thm. 2.2 and 2.3]{LN} to arbitrary ground fields in the finite group scheme setting. Namely, Theorems \ref{fund2} and \ref{keycoro} provide practical criteria for the equality in the equivariant Grothendieck ring $\mathscr M_k^G$ of $k$-varieties endowed with good action of a finite group $k$-scheme $G$.  

We start with formal $R$-schemes topologically of finite type $\mathfrak X$ (stft formal schemes, for short) which is endowed with a good action of a finite group $k$-scheme $G$. Such a formal scheme $\mathfrak X$ corresponds to an inductive system of $R_n$-schemes $X_n$ with induced $G$-action. In Proposition \ref{action-on-Greenberg}, we show that the Greenberg functor brings the $G$-action on $\mathfrak X$ to a good $G$-action on each $k$-variety $\Gr_n(X_n)$ (resp. on $\Gr(\mathfrak X)$) such that the construction is functorial. The combination of Theorem \ref{keycoro} and Proposition \ref{action-on-Greenberg} gives rise to a $G$-equivariant motivic measure $\mu_{\mathfrak X}^G$ with value in $\mathscr M_{\mathfrak X_s}^G$ on $G$-invariant stable cylinders of $\Gr(\mathfrak X)$ (cf. Proposition \ref{measure}). This fact allows us to define the motivic $G$-integral $\int_{\mathscr A}\L^{-\alpha}d\mu_{\mathfrak X}^G\in \mathscr M_{\mathfrak X_s}^G$ of a motivic function $\L^{-\alpha}$ on such a cylinder $\mathscr A$, where $\L$ is the class of $\mathbb A_k^1\times_k\mathfrak X_s\to \mathfrak X_s$ in $\mathscr M_{\mathfrak X_s}^G$ (cf. Section \ref{Sec2.1}) and $\alpha: \mathscr A\to \mathbb Z \cup \{\infty\}$ is a naively exponentially $G$-integrable function (cf. Definition \ref{def-stft}). We revisit the change of variables formula of Sebag \cite[Thm. 8.0.5]{Se} in the $G$-equivariant version, see Theorem \ref{change-of-variable}, in which the identity of integrals holds in $\mathscr M_{\mathfrak X_s}^G$.
Furthermore, if $\mathfrak X$ is flat stft generically smooth, then for any gauge form $\omega$ on the generic fiber $\mathfrak X_{\eta}$ of $\mathfrak X$, we have a $\mathbb Z$-value function $\ord_{\varpi,\mathfrak X}(\omega)$ on $\Gr(\mathfrak X)\setminus \Gr(\mathfrak X_{\sing})$, which is naively exponentially $G$-integrable due to Lemma \ref{orderomega}. In this case, we denote by $\int_{\mathfrak X}|\omega|$ the motivic $G$-integral of $\L^{-\ord_{\varpi, \mathfrak X}(\omega)}$ on $\Gr(\mathfrak X)\setminus \Gr(\mathfrak X_{\sing})$, and also put $\int_{\mathfrak X_{\eta}}|\omega|:=\int_{\mathfrak X_s}\int_{\mathfrak X}|\omega|$ in $\mathscr M_k^G$. 

The major purpose of the present article is to develop a theory of $G$-equivariant motivic integration of gauge forms on {\it special formal $R$-schemes}, which is an equivariant version of Nicaise's integration \cite{Ni2}, also a natural upgrade of our previous work \cite{LN}. It is important to notice that we shall only consider {\it adic} $G$-actions on special formal $R$-schemes, and only consider $G$-equivariant {\it adic} morphisms of special formal $R$-schemes, which guarantee the existence of induced $G$-equivariant $k$-morphisms at the reduction level. Remark that every morphism of stft formal $R$-schemes is automatically an adic morphism. To construct motivic integration on a special formal scheme Nicaise \cite{Ni2} uses its dilatation that always exists with respect to a coherent ideal sheaf containing $\varpi$. For any flat special formal $R$-scheme endowed with a good adic $G$-action, we prove in Proposition \ref{G-dilatation} that under certain conditions on a coherent ideal sheaf $\mathcal I$ on a special formal scheme $\mathfrak X$, it admits a so-called $G$-dilatation $\pi: \mathfrak U\to \mathfrak X$. When $\mathcal I$ defines $\mathfrak X_0$, $\mathfrak U$ is a flat stft formal $R$-scheme, and for any gauge form $\omega$ on $\mathfrak X_{\eta}$, the differential form $\pi_{\eta}^*\omega$ is also a gauge form on $\mathfrak U_{\eta}$. It is natural to define the motivic $G$-integrals of $\omega$ on $\mathfrak X$ and on $\mathfrak X_{\eta}$ as follows 
\begin{align*}
\int_{\mathfrak X}|\omega|:= {\pi_s}_!\int_{\mathfrak U}|\pi_{\eta}^*\omega|\in \mathscr M_{\mathfrak X_0}^G\ \ \text{and}\  \int_{\mathfrak X_{\eta}}|\omega|:=\int_{\mathfrak U_{\eta}}|\omega|\in \mathscr M_k^G.
\end{align*}
The following proposition is the most important technical result, which asserts the existence of a $G$-N\'eron smoothening for a special formal scheme.

\begin{prop}[Proposition \ref{G-smoothening}]
Let $G$ be a smooth finite group $k$-scheme. Then every flat generically smooth special formal $R$-scheme $\mathfrak X$ endowed with a good adic $G$-action admits a $G$-N\'eron smoothening $\mathfrak h\colon \mathfrak Y \to\mathfrak X$.
\end{prop}

This basic tool together with Proposition \ref{dilatation_stft} guarantees  that the notion of $G$-equivariant motivic integration for special formal schemes is an obvious extension of the one for stft formal schemes. It also promotes its crucial effects on the main results of the present article. Indeed, by Proposition \ref{dilatation_stft}, we realize some basic properties of the $G$-equivariant motivic integration, such as the compatibility with formal completion (Proposition \ref{base_change}) and the additivity (Corollary \ref{int-additive}). Moreover, if $\omega$ is an $\mathfrak X$-{\it bounded} gauge form, we obtain in Proposition \ref{special-connected} a $G$-equivariant version of \cite[Prop. 5.14]{Ni2} on the expression of the motivic integral $\int_{\mathfrak X}|\omega|$ in terms of the connected components of $\mathfrak X_0$. The most significant result on the $G$-equivariant motivic integration for special formal schemes that we obtain, also under the support of Proposition \ref{G-smoothening}, is the special $G$-equivariant change of variables formula (see Theorem \ref {changeofvariables}).

\begin{thm}[Theorem \ref {changeofvariables}]
Let $G$ be a smooth finite group $k$-scheme. Let $\mathfrak X$ and $\mathfrak Y$ be generically smooth special formal $R$-schemes endowed with good adic actions of $G$, and let $\mathfrak h\colon \mathfrak Y \to\mathfrak X$ be an adic $G$-equivariant morphism of formal $R$-schemes such that  the induced morphism $\mathfrak Y_{\eta} \to \mathfrak X_{\eta}$ is an open embedding and $\mathfrak Y_{\eta}(K^{sh})=\mathfrak X_{\eta}(K^{sh})$. If $\omega$ is a gauge form on $\mathfrak X_{\eta}$, then 
$$\int_{\mathfrak X}|\omega|={\mathfrak h_0}_!\int_{\mathfrak Y}| \mathfrak h_{\eta}^*\omega|\quad \text{in} \ \ \mathscr M_{\mathfrak X_0}^G.$$
\end{thm}

An important application of Theorem \ref {changeofvariables} and Proposition \ref{special-connected} is the proof of rationality of the (monodromic) motivic volume Poincar\'e series. Assume that $\mathfrak X$ is generically smooth. Then, for any $n\in \mathbb N^*$, $\mathfrak X(n):=\mathfrak X\times_RR(n)$ is a generically smooth special formal $R(n)$-scheme and it is naturally endowed with an adic good $\mu_n$-action (see Lemma \ref{mu_n-action}). Let $\omega(n)$ be the gauge form on $\mathfrak X(n)_{\eta}$ induced by $\omega$. We apply the $\mu_n$-equivariant motivic integration developed in this article to the formal schemes $\mathfrak X(n)$ and consider the motivic volume Poincar\'e series with respect to a gauge form $\omega$ on $\mathfrak X_{\eta}$:
$$
P(\mathfrak X,\omega; T):=\sum_{n\geq 1}\Big(\int_{\mathfrak X(n)}|\omega(n)|\Big)T^n\ \in \mathscr M_{\mathfrak X_0}^{\hat\mu}[[T]].
$$
Assume, in addition, that $\mathfrak X$ is flat and admits a tame resolution of singularities $\mathfrak h$, and that $\omega$ is $\mathfrak X$-bounded. Then we obtain in Theorem \ref{int_Xm} a description of $\int_{\mathfrak X(n)}|\omega(n)|$ in terms of $\mathfrak h$ provided $n$ is prime to the characteristic exponent of $k$; this result is an equivariant version of \cite[Thm. 7.12]{Ni2}. If $k$ is of characteristic zero, then $\mathfrak X$ has a (tame) resolution of singularities, therefore together with Theorem \ref{int_Xm} we prove the rationality of $P(\mathfrak X,\omega; T)$ (Corollary \ref{poincare}). Using the map $\lim_{T\to \infty}$ in \cite{DL1} we define $\MV(\mathfrak X;\widehat{K^s}):=-\lim_{T\to\infty}\L^{\dim_R\mathfrak X}P(\mathfrak X,\omega; T)\in \mathscr M_{\mathfrak X_0}^{\hat\mu}$ with $\dim_R\mathfrak X$ the relative dimension of $\mathfrak X$, which is called the motivic volume of $\mathfrak X$.

Let $f$ be a formal power series in $k\{x\}[[y]]$ such that the series $f(x,0)$ is non-constant, with $x=(x_1,\dots, x_m)$ and $y=(y_1,\dots,y_{m'})$. Denote by $\mathfrak X_f$ the formal completion of $\Spf(k\{x\}[[y]])$ along $(f)$, which is a special formal $R$-scheme of pure relative dimension $m+m'-1$ with the reduction $(\mathfrak X_f)_0=\Spec k[x]/(f(x,0))$. Denote by $\omega/df$ the Gelfand-Leray form associated to a top differential form $\omega$ on $\mathfrak X_f$ considered as a formal scheme over $k$ (we may take $\omega=dx_1\wedge\cdots\wedge dx_d$). Let $k$ be of characteristic zero. Using tame resolution of singularity we obtain the formulas for $\int_{\mathfrak X_f(n)}|(\omega/df)(n)|$ and $\MV(\mathfrak X_f;\widehat{K^s})$ (cf. Proposition \ref{mzf-fseries}), which extend Denef-Loeser's formulas for the contact locus $[\mathscr X_n(f)]$ and the motivic nearby cycles $\mathscr S_f$ of a polynomial $f$ (cf. Corollary \ref{comparezeta}). This result allows us to define the so-called motivic nearby cycles and motivic Milnor fiber of a series in $k\{x\}[[y]]$.

Finally, we borrow Section \ref{last-ss} of the present article to discuss an observation concerning motivic Milnor fiber of complex analytic function germs. The original approach of Denef-Loeser to motivic Milnor fiber (cf. \cite{DL1, DL5}) is only for regular function germs (i.e. polynomials) since their theory of motivic integration works only with algebraic varieties and usual schemes. Hence, it is surprising there were unproven declarations that the concept of motivic Milnor fiber for complex analytic function germs is only an obvious extension of that of Denef-Loeser. To provide a solution to this extension we recommend Conjectures \ref{conj1} and \ref{conj2}, which may be expected as a bridge connecting equivariant motivic integration for special formal schemes and singularity theory. 

\begin{ack}
This article has been written during many visits of the authors to Vietnam Institute for Advanced Studies in Mathematics between 2018 and 2025. The authors thank sincerely the institute for their hospitality and valuable supports. The first author would also like to acknowledge support from the ICTP through the Associates Programme (2020-2025). 
\end{ack}


\section{Preliminaries}\label{Prel}
\subsection{Equivariant Grothendieck rings of varieties} \label{Sec2.1}
Let $k$ be a field, and let $S$ be a $k$-variety. By an $S$-variety we mean a $k$-variety $X$ together with a morphism $X\to S$. As in \cite{DL1} and \cite{DL2}, let $\Var_S$ denote the category of $S$-varieties and $K_0(\Var_S)$ its Grothendieck ring. By definition, $K_0(\Var_S)$ is the quotient of the free abelian group generated by all $S$-isomorphism classes $[X\to S]$ in $\Var_S$ modulo the relations
$$[X\to S]=[X_{\red}\hookrightarrow X\to S]$$
and 
$$[X\to S]=[Y\to S]+[X\setminus Y\to S],$$
where $X_{\red}$ is the reduced subscheme of $X$, and $Y$ is a Zariski closed subvariety of $X$. Together with fiber product over $S$, $K_0(\Var_S)$ is a commutative ring with unity $1=[\Id\colon S\to S]$. Put $\L=[\mathbb A_k^1\times_k S\to S]$. Denote by $\mathscr M_S$ the localization of $K_0(\Var_S)$ which makes $\L$ invertible.

Let $X$ be a $k$-variety, and let $G$ be an algebraic group. An action of $G$ on $X$ is said to be {\it good} if every $G$-orbit is contained in an affine open subset of $X$. We fix a good action of $G$ on the $k$-variety $S$ (where we may choose the trivial action). By definition, the $G$-equivariant Grothendieck group $K_0^G(\Var_S)$ of $G$-equivariant morphisms of $k$-varieties $X\to S$, where $X$ is endowed with a good $G$-action, is the quotient of the free abelian group generated by the $G$-equivariant isomorphism classes $[X\to S,\sigma]$, where $X$ is a $k$-variety endowed with a good $G$-action $\sigma$ and $X\to S$ is a $G$-equivariant morphism of $k$-varieties, modulo the relations
$$[X\to S,\sigma]=[Y\to S,\sigma|_Y]+[X\setminus Y\to S,\sigma|_{X\setminus Y}]$$
for $Y$ being $\sigma$-invariant Zariski closed in $X$, and 
\begin{align*}
[X\times_k\mathbb A_k^n\to S,\sigma]=[X\times_k\mathbb A_k^n\to S,\sigma']
\end{align*}
if $\sigma$ and $\sigma'$ lift the same good $G$-action on $X$. As above, we have the commutative ring with unity structure on $K_0^G(\Var_S)$ by fiber product, where $G$-action on the fiber product is through the diagonal $G$-action. Denote by $\mathscr M_S^G$ the localization ring $K_0^G(\Var_S)[\L^{-1}]$, where we view $\L$ as the class of $\mathbb A_k^1\times_k S\to S$ endowed with the trivial action of $G$. 

Denote by $\hat\mu$ the limit of the projective system of group schemes $\mu_n=\Spec \left(k[\xi]/(\xi^n-1)\right)$ with transition morphisms $\mu_{mn}\to \mu_n$ sending $\lambda$ to $\lambda^m$. Define $K_0^{\hat\mu}(\Var_S)=\varinjlim K_0^{\mu_n}(\Var_S)$ and $\mathscr M_S^{\hat\mu}=K_0^{\hat\mu}(\Var_S)[\L^{-1}]$.  

Let $G$ be an algebraic group. Let $f\colon S\to S'$ be a $G$-equivariant morphism of $k$-varieties with a good $G$-action. Denote by $f^*$ the ring homomorphism $K_0^G(\Var_{S'})\to K_0^G(\Var_S)$ induced from the fiber product (the pullback morphism), and by $f_!$ the $K_0(\Var_k)$-linear homomorphism $K_0^G(\Var_{S})\to K_0^G(\Var_{S'})$ defined by the composition with $f$ (the push-forward morphism). The pullback morphism induces a unique morphism of localizations $f^*\colon \mathscr{M}_{S'}^G\to \mathscr{M}_S^G$, the push-forward morphism induces a unique $\mathscr M_k$-linear morphism $f_!\colon \mathscr{M}_S^G\to \mathscr{M}_{S'}^G$ by sending $a\L_S^{-n}$ to $(f_!a)\L_{S'}^{-n}$ for any $a$ in $K_0^G(\Var_{S})$ and any $n\in \mathbb N$. When $S'$ is $\Spec k$, we usually write $\int_S$ instead of $f_!$.


\subsection{Equivariant piecewise trivial fibrations}
Let $X$, $Y$ and $F$ be algebraic $k$-varieties endowed with a good action of an algebraic group $G$. Let $A$ and $B$ be $G$-invariant constructible subsets of $X$ and $Y$, respectively. Let $f\colon X\to Y$ be a $G$-equivariant morphism such that $f(A)\subseteq B$. The restriction $f\colon A\to B$ is called a {\it $G$-equivariant piecewise trivial fibration with fiber $F$} if there exists a stratification of $B$ into finitely many $G$-invariant locally closed subsets $B_i$ such that $f^{-1}(B_i)$ is a $G$-invariant constructible subset of $A$ and $f^{-1}(B_i)$ is $G$-equivariant isomorphic to $B_i\times_k F$ with respect to the diagonal action of $G$ on $B_i\times_k F$, and such that, over $B_i$, $f$ equals the projection $B_i\times_k F\to B_i$. 

\medskip
For a morphism of algebraic $k$-varieties $X\to Y$ and an immersion $S\to Y$, we write $X_S$ for the fiber product $X\times_YS$. If $Y$ is endowed with a good $G$-action, then for $y$ in $Y$, the {\em stabilizer subgroup $G_y$} of $G$ over $y$ is the subgroup of elements in $G$ fixing $y$. 

\begin{theorem}\label{fund2} 
Let $G$ be a finite group $k$-scheme. Suppose that $X$ and $Y$ are algebraic $k$-varieties endowed with a good $G$-action and that $f\colon X\to Y$ is a $G$-equivariant morphism. Then $f$ is a $G$-equivariant piecewise trivial fibration if and only if there is an algebraic $k$-variety $F$ endowed with a good $G$-action such that for every $y$ in $Y$, there is a $G_y$-equivariant isomorphism of algebraic $\kappa(y)$-varieties $X_y \stackrel{\cong}{\longrightarrow} F\times_k \kappa(y)$.
\end{theorem}

\begin{proof}
Similarly as in the proof of \cite[Thm. 2.2]{LN}.
\end{proof}

\begin{theorem}\label{keycoro} 
Let $G$ be a finite group $k$-scheme. Suppose that $X$ and $Y$ are algebraic $k$-varieties endowed with a good $G$-action and that $f\colon X\to Y$ is a $G$-equivariant morphism. If there is an $n$ in $\mathbb N$ such that for every $y$ in $Y$, there is an isomorphism of algebraic $\kappa(y)$-varieties $X_y\cong \mathbb{A}^n_{\kappa(y)}$, then $[X]=[Y]\mathbb{L}^n$ in $K_0^G(\Var_k)$.
\end{theorem}

\begin{proof}
Similarly as in the proof of \cite[Thm. 2.3]{LN}.
\end{proof}


\subsection{Rational series}\label{Hadamard}
Let $\mathscr M$ be a commutative ring with unity containing $\L$ and $\L^{-1}$, and let $\mathscr M[[T]]$ be the set of formal power series in one variable $T$ with coefficients in $\mathscr M$, which is a ring and also a $\mathscr M$-module with respect to usual operations for series. Denote by $\mathscr M[[T]]_{\sr}$ the submodule of $\mathscr M[[T]]$ generated by 1 and by finite products of terms $\frac{\L^aT^b}{1-\L^aT^b}$ for $(a,b)$ in $\mathbb{Z}\times\mathbb N^*$. Any element of $\mathscr M[[T]]_{\sr}$ is called a {\it rational} series. It is proved in \cite{DL1} that there exists a unique $\mathscr M$-linear morphism $\lim\limits_{T\to\infty}: \mathscr M[[T]]_{\sr}\to \mathscr M$ such that $\lim\limits_{T\to\infty}\frac{\L^aT^b}{1-\L^aT^b}=-1$
for any $(a,b)$ in $\mathbb{Z}\times\mathbb N^*$. 

The Hadamard product of two formal power series $p=\sum_{n\geq 1}p_nT^n$ and $q=\sum_{n\geq 1}q_nT^n$ in $\mathscr M[[T]]$ is a formal power series in $\mathscr M[[T]]$ defined as $p\ast q:=\sum_{n \geq 1}p_n q_nT^n$. This product is commutative, associative, and has the unity $\sum_{n\geq 1}T^n$. It also preserves the rationality as seen in the following lemma.

\begin{lemma}[Looijenga \cite{Loo}]\label{Lem2}
If $p(T)$ and $q(T)$ are rational series in $\mathscr M[[T]]$, then $p(T)\ast q(T)$ is also a rational series, and in this case,
$$\lim_{T\to\infty}p(T)\ast q(T)=-\lim_{T\to\infty}p(T) \cdot \lim_{T\to\infty}q(T).$$
\end{lemma}


\section{Equivariant motivic integration on stft formal schemes}
Let $R$ be a complete discrete valuation ring with fraction field $K$ and residue field $k$. Assume that $K$ and $k$ have the same characteristic. Fix a $k$-algebra structure on $R$ so that the composition $k\hookrightarrow R\to k$ is $\mathrm{Id}_k$. Let $\varpi\in R$ be a uniformizing parameter, which is fixed throughout this article, and let $R_n=R/(\varpi^{n+1})$. We also fix a separable closure $K^s$ of $K$, denote respectively by $K^t$ and $K^{sh}$ the tame closure and strict henselization of $K$ in $K^s$.

\subsection{Formal schemes topologically of finite type with action}\label{section3.1}
According to EGA1, Ch. 0, 7.5.3, an adic $R$-algebra $A$ is {\it of finite type} if $(\varpi)A$ is an ideal of definition of $A$ and $A/(\varpi)A$ is a $k$-algebra of finite type. For any $n\in \mathbb N$, let $R\{x_1,\dots,x_n\}$ denote the $R$-algebra of restricted power series in $n$ variables, namely, 
$$R\{x_1,\dots,x_n\}=\varprojlim_{\ell}  (R/(\varpi^{\ell}))[x_1,\dots,x_n].$$
Clearly, $R\{x_1,\dots,x_n\}$ is a Noetherian ring and the definition of $R\{x_1,\dots,x_n\}$ is independent of the choice of $\varpi$. It is a fact that $A$ is of finite type if it is topologically $R$-isomorphic to a quotient algebra of  $R\{x_1,\dots,x_n\}$ for some $n\in \mathbb N$. 

For any adic ring $A$, we denote by $\Spf A$ the set of all open prime ideals of $A$, which has a structure of a locally ringed space and is called {\it the formal spectrum of $A$}. A Noetherian adic formal scheme is a locally ringed space which is locally isomorphic to the formal spectrum of a Noetherian adic ring.

\begin{definition}
A {\it formal $R$-scheme topologically of finite type} is a Noetherian adic formal scheme which is a finite union of affine formal schemes of the form $\Spf A$ with $A$ an adic $R$-algebra of finite type.  
\end{definition}

If $\mathfrak X$ is a separated formal $R$-scheme topologically of finite type over $R$, it will be abbreviated as {\it stft} formal $R$-scheme. Such an $\mathfrak X$ is nothing but the inductive limit of the $R_n$-schemes 
\begin{align}\label{Xnstft}
X_n=(\mathfrak X, \mathcal O_{\mathfrak X}\otimes_RR_n),
\end{align} 
and the transition morphisms $X_n\to X_m$ as $R_m$-schemes (for $n\leq m$) are induced from the truncated map $R_m\to R_n$. Using the morphism $X_n\to X_m$ we have 
\begin{align}\label{XmXn}
X_n\cong X_m\times_{\Spec R_m}\Spec R_n.
\end{align}
Clearly, the category of formal $R$-schemes topologically of finite type admits fiber products. 

\begin{definition}\label{fmorphisms}
A {\it morphism} of stft formal $R$-schemes $\mathfrak{Y}\to\mathfrak{X}$ is a morphism between the underlying locally topologically ringed spaces over $R$. This morphism is said to be {\it locally of finite type} if locally it is a morphism of the form $\Spf B\to \Spf A$, where the corresponding $R$-morphism $A\to B$ is of finite type.
\end{definition}

\begin{notation}
For any Noetherian adic formal scheme $\mathfrak X$ we denote by $\mathfrak X_s$ the {\it special fiber} $\mathfrak X\times_Rk$ of $\mathfrak X$, which is a formal $k$-scheme, and denote by $\mathfrak X_0$ the closed subscheme of $\mathfrak X$ defined by the largest ideal of definition of $\mathfrak X$, called the {\it reduction} of $\mathfrak X$, which is a reduced noetherian scheme. If $\mathfrak X$ is a stft formal $R$-scheme, then $\mathfrak X_s=X_0$ (cf. eq. (\ref{Xnstft})), which is a separated $k$-scheme of finite type, and has the property $\mathfrak X_0=(\mathfrak X_s)_{\red}$.
\end{notation}

Assume that $\mathfrak X=\Spf A$, where $A$ is an $R$-algebra of finite type. The tensor product $A\otimes_RK$ is then a $K$-affinoid algebra in the sense of Tate \cite{Tate}. The rigid $K$-variety $\Spm(A\otimes_RK)$ is called {\it the generic fiber of $\mathfrak X$} and denoted by $\mathfrak X_{\eta}$. It is shown that the correspondence 
$$\Spf A\mapsto \Spm(A\otimes_RK)$$ 
is functorial, and that it can be extended to any stft formal $R$-schemes $\mathfrak X\mapsto \mathfrak X_{\eta}$ by glueing procedure along open coverings of $\mathfrak X$ (see \cite{Berthelot} or \cite{deJ}). The rigid $K$-variety $\mathfrak X_{\eta}$ is separated and quasi-compact. The formal scheme $\mathfrak X$ is called {\it generically smooth} if its generic fiber $\mathfrak X_{\eta}$ is a smooth rigid $K$-variety.


\begin{definition}\label{action-k}
Let $G$ be a finite group $k$-scheme with $m_G\colon G\times_kG\to G$ the multiplication and $e_G\in G(k)$ the neutral element. A {\it (left) $G$-action} on a formal $k$-scheme $\mathfrak X$ is a morphism of formal $k$-schemes 
$$\theta \colon G\times_k \mathfrak X \to \mathfrak X, \quad (g,x)\mapsto g\cdot x$$
such that the composite map $\theta\circ (e_G\times\Id): \mathfrak X\cong \Spec k\times \mathfrak X \to G\times \mathfrak X \to \mathfrak X$ is the identity and that the below diagram commutes:
$$\begin{CD}
G\times_kG\times_k \mathfrak X @>\Id\times\theta>> G\times_k \mathfrak X\\
@Vm_G\times \Id VV @VV\theta V\\
G\times_k \mathfrak X @>\theta >> \ \! \mathfrak X.
\end{CD}$$	
Assume that $(\mathfrak X,\theta)$ and $(\mathfrak Y,\theta')$ are formal $k$-schemes endowed with $G$-action. Then a morphism $\mathfrak f\colon \mathfrak X \to \mathfrak Y$ is called {\it $G$-equivariant} if the following diagram commutes:
$$\begin{CD}
G\times_k \mathfrak X @>\theta>> \mathfrak X\\
@V\Id_G\times \mathfrak f VV @VV \mathfrak f V\\
G\times_k \mathfrak Y @>\theta'>> \ \! \mathfrak Y.
\end{CD}$$
\end{definition}

Since $\Spf R=\varinjlim\Spec R_n$, any $G$-action $\sigma$ on $\Spf R$ induces a unique $G$-action $\sigma_n$ on $\Spec R_n$ such that the obvious $k$-morphisms $\iota_n\colon \Spec R_n\to \Spf R$ and $\iota_{n,m}\colon \Spec R_n\to \Spec R_m$ (for $n\leq m$) are $G$-invariant. From now on, we fix a $G$-action $\sigma$ on $\Spf R$.

\begin{definition}\label{action-R}
Let $\mathfrak X$ be a formal $R$-scheme. A {\it $G$-action} on $\mathfrak X$ is a $G$-action on the formal $k$-scheme $\mathfrak X$ (with the $k$-scheme structure induced from $k\hookrightarrow R$) such that the structural morphism $\mathfrak X\to \Spf R$ viewed as a morphism of formal $k$-schemes is $G$-equivariant. A $G$-action on $\mathfrak X$ is called {\it good} if any orbit of it is contained in an affine open formal subscheme of $\mathfrak X$. 
\end{definition}

Let $\mathfrak X$ be a stft formal $R$-scheme, which is the inductive limit of $R_n$-scheme $X_n$ mentioned above. Then a $G$-action $\theta$ on $\mathfrak X$ induces a unique $G$-action $\theta_n$ on $X_n$ such that the obvious $k$-morphisms $\rho_n\colon X_n\to \mathfrak X$ and $\rho_{n,m}\colon X_n\to X_m$ (for $n\leq m$) are $G$-equivariant. If $\mathfrak f\colon \mathfrak Y \to \mathfrak X$ is a morphism of stft formal $R$-schemes, then the $G$-equivariance of $\mathfrak f$ induces $G$-equivariant morphisms $f_n\colon Y_n \to X_n$ compatible with $\rho_n$ and $\rho_{n,m}$. In particular, $\theta_0$ is the $G$-action on the special fiber $\mathfrak X_s=X_0$ induced from $\theta$, and the $G$-equivariant morphism of $k$-schemes $\mathfrak f_s=f_0\colon \mathfrak Y_s\to \mathfrak X_s$ is induced from the $G$-equivariant morphism of formal $R$-schemes $\mathfrak f$.

\subsection{Greenberg spaces of stft formal schemes}
Let $\mathfrak X$ be a stft formal $R$-scheme, which is the inductive limit of the $R_n$-schemes $X_n$ described in (\ref{Xnstft}). Since $R$ is of equal characteristic, the functor defined locally by
$$\Spec A\mapsto \Hom_{R_n}(\Spec\left(R_n\otimes_kA\right),X_n)$$ 
from the category of $k$-schemes to the category of sets is represented by a $k$-scheme $\Gr_n(X_n)$ of finite type such that, for any $k$-algebra $A$,
$$\Gr_n(X_n)(A)=X_n(R_n\otimes_kA),$$  
see Greenberg \cite{Gr}. For $n\leq m$ and $\gamma\colon \Spec(R_m\otimes_kA)\to X_m$ in $\Gr_m(X_m)(A)$, tensoring with $\Spec R_n$ over $\Spec R_m$ we get an element
\begin{align}\label{gammatilde}
\widetilde\gamma:=\gamma \times_{R_m} \Id\colon \Spec(R_n\otimes_kA)\to X_m\times_{\Spec R_m}\Spec R_n\cong X_n
\end{align}
in $\Gr_n(X_n)(A)$, where $\gamma \times_{R_m} \Id$ stands for $\gamma \times_{\Spec R_m} \Id_{\Spec R_n}$. Thus, the correspondence $\gamma\mapsto \widetilde\gamma$ gives a map $\Gr_m(X_m)(A)\to \Gr_n(X_n)(A)$, and by the functoriality on $A$ we have a canonical morphism of $k$-schemes 
$$\pi_n^m=(\pi_{\mathfrak X})_n^m: \Gr_m(X_m)\to \Gr_n(X_n).$$
We obtain a projective system $\left\{\big(\Gr_n(X_n);\pi_n^m\big) \mid n\leq m\right\}$ in the category of separated $k$-schemes of finite type. Put $\Gr(\mathfrak X):=\varprojlim_n  \Gr_n(X_n)$, which exists in the category of $k$-schemes. Writing $\pi_n=\pi_{\mathfrak X,n}$ for the canonical projection $\Gr(\mathfrak X)\to \Gr_n(X_n)$ we have the following commutative diagram, for $n\leq m$,
\begin{displaymath}
\xymatrixcolsep{5pc}\xymatrix{
	\Gr(\mathfrak X) \ar[r]^{\pi_m} \ar[rd]_{\pi_n}
	&\Gr_m(X_m)\ar[d]^{\pi_n^m}\\
	&\Gr_n(X_n).
}
\end{displaymath}
The essential properties of the space $\Gr(\mathfrak X)$ are shown in Sebag \cite{Se} and Loeser-Sebag \cite{LS}. In particular, there is a fact that every point $x\in \Gr(\mathfrak X)$ with residue field $\kappa(x)$ corresponds a morphism $\gamma\colon \Spf \big(R\widehat{\otimes}_k\kappa(x)\big)\to \mathfrak X$. 

\begin{definition}
For any stft formal $R$-scheme $\mathfrak X$, the $k$-scheme $\Gr(\mathfrak X)$ defined previously is called {\it the Greenberg space of $\mathfrak X$}.
\end{definition}

Let $\mathfrak f\colon \mathfrak Y\to \mathfrak X$ be a morphism of stft formal $R$-schemes. Then $\mathfrak f$ is the injective limit of morphisms of compatible $R_n$-schemes $f_n\colon Y_n\to X_n$. The morphisms $f_n$ correspond to the maps 
$$Y_n(R_n\otimes_kA)\to X_n(R_n\otimes_kA)$$
for any $k$-algebra $A$. Since this construction is functorial in $A$, it defines $k$-morphisms of schemes 
$$\Gr_n(f_n)\colon \Gr_n(Y_n)\to \Gr_n(X_n)$$ 
such that the diagram
$$\begin{CD}
\Gr_m(Y_m) @>\Gr_m(f_m)>> \Gr_m(X_m)\\
@V(\pi_{\mathfrak Y})_n^m VV @VV(\pi_{\mathfrak X})_n^m V\\
\Gr_n(Y_n) @>\Gr_n(f_n)>> \Gr_n(X_n)
\end{CD}$$
commutes for every $n\leq m$ in $\mathbb N^*$. Taking the projective limit we get a unique morphism of $k$-scheme 
$$\Gr(\mathfrak f)\colon \Gr(\mathfrak Y)\to \Gr(\mathfrak X)$$ 
such that the following diagram commutes:
$$\begin{CD}
\Gr(\mathfrak Y) @>\Gr(\mathfrak f)>> \Gr(\mathfrak X)\\
@V\pi_{\mathfrak Y,n} VV @VV\pi_{\mathfrak X,n} V\\
\Gr_n(Y_n) @>\Gr_n(f_n)>> \ \!\Gr_n(X_n).
\end{CD}$$

\subsection{$G$-actions on Greenberg spaces}
As above, we consider a stft formal $R$-scheme $\mathfrak X$, with $X_n=(\mathfrak X, \mathcal O_{\mathfrak X}\otimes_RR_n)$. Let $G$ be a finite group $k$-scheme. Let $\theta$ be a good $G$-action on $\mathfrak X$, i.e. a $G$-action $\theta\colon G\times_k \mathfrak X\to \mathfrak X$ on formal $k$-scheme $\mathfrak X$ (the formal $k$-scheme structure on $\mathfrak X$ is induced from the inclusion $k\hookrightarrow R$) with the structural morphism $\mathfrak X\to \Spf R$ being $G$-equivariant. 

\begin{proposition}\label{action-on-Greenberg}
For any stft formal $R$-scheme $\mathfrak X$, there exist good $G$-actions 
$$\Gr(\theta)\colon G\times_k \Gr(\mathfrak X) \to \Gr(\mathfrak X)$$ 
and 
$$\Gr_n(\theta_n)\colon G\times_k\Gr_n(X_n) \to \Gr_n(X_n),$$ 
for all $n\in \mathbb N$, satisfying the following conditions: 
\begin{itemize}
\item[(i)] $\Gr_0(\theta_0)=\theta_0$, i.e. $\Gr(\theta)$ and $\theta$ induce the same action on $\mathfrak X_s=X_0$;
	
\item[(ii)] The $k$-morphisms $\pi_n$ and $\pi_n^m$ are $G$-equivariant for every $n\leq m$ in $\mathbb N$;
	
\item[(iii)] If $\mathfrak f\colon \mathfrak X\to \mathfrak Y$ is a $G$-equivariant morphism of stft formal $R$-schemes endowed with $G$-action, then the induced morphisms $\Gr(\mathfrak f)$ and $\Gr_n(f_n)$ are $G$-equivariant. 
\end{itemize}
\end{proposition}

\begin{proof}
Let $A$ be an arbitrary $k$-algebra. Then the $G$-action $\sigma_n$ on $\Spec R_n$ induces naturally a $G$-action on $\Spec(R_n\otimes_kA)$, which is by abuse of notation also denoted by $\sigma_n$. Consider the following map 
$$\Gr_n(\theta_n)\colon G(A)\times X_n(R_n\otimes_kA)\to X_n(R_n\otimes_kA)$$
in which, for $g\in G(A)$ and $\gamma\in X_n(R_n\otimes_kA)$, the morphism $g\cdot \gamma:=\Gr_n(\theta_n)(g,\gamma)$ is given by the commutative diagram
$$\begin{CD}
\Spec(R_n\otimes_kA) @>i>>\Spec A\times_k \Spec(R_n\otimes_kA)\\
@Vg\cdot\gamma VV @VVg\times \gamma V\\
X_n @<\theta_n<< G\times_k X_n.
\end{CD}$$
Let $e$ be the neutral element of $G(A)$. Then $e\cdot\gamma=\gamma$ for any $\gamma\in X_n(R_n\otimes_kA)$. For $g, h\in G(A)$ and $\gamma\in X_n(R_n\otimes_kA)$, we have
\begin{align*}
h\cdot(g\cdot \gamma)&=\theta_n \circ (h\times g\cdot\gamma)\circ i,\\
(hg)\cdot\gamma&=\theta_n\circ (hg\times\gamma)\circ i.
\end{align*}
Since $\theta_n$ is a $G$-action on $X_n$ we have $\theta_n \circ (h\times g\cdot\gamma)=\theta_n\circ (hg\times\gamma)$. Hence $h\cdot(g\cdot \gamma)=(hg)\cdot\gamma$, from which $\Gr_n(\theta_n)$ is an action of the abstract group $G(A)$ on $\Gr_n(X_n)(A)$.  

The construction is functorial because for every homomorphism of $k$-algebras $A\to B$ the below diagram commutes:
$$\begin{CD}
\Spec(R_n\otimes_kB) @>>>\Spec B\times_k \Spec(R_n\otimes_kB)\\
@VVV @VVV\\
\Spec(R_n\otimes_kA) @>>>\Spec A\times_k \Spec(R_n\otimes_kA).
\end{CD}$$
Therefore, by the Yoneda lemma, we obtain a $G$-action
$$\Gr_n(\theta_n)\colon G\times_k \Gr_n(X_n)\to \Gr_n(X_n)$$
on $\Gr_n(X_n)$, which relates to $\sigma_n$ and $\theta_n$. It is obvious from the construction that $\Gr_0(\theta_0)=\theta_0$. Moreover, because for $n\leq m$ the diagram
$$\begin{CD}
G\times_k X_n @>\theta_n>> X_n\\
@V\Id\times \rho_{n,m}VV @VV\rho_{n,m} V\\
G\times_k X_m @>\theta_m>> X_m
\end{CD}$$
commutes, it follows by the construction that $\pi_n^m\colon \Gr_m(X_m)\to \Gr_n(X_n)$ are $G$-equivariant. Clearly, if $\theta_n$ is good, so is $\Gr_n(\theta_n)$.
	
Now, putting $\Gr(\theta):= \varprojlim \Gr_n(\theta_n)$ we get a good $G$-action $\Gr(\theta)\colon G\times_k \Gr(\mathfrak X)\to \Gr(\mathfrak X)$, which guarantees that the canonical morphisms $\pi_n\colon \Gr(\mathfrak X)\to \Gr_n(X_n)$ are $G$-equivariant. More concretely, for any $g$ in $G$, any $k$-algebra $A$, $g\in G(A)$ and $\gamma$ in $\Gr(\mathfrak X)(A)$, we have
$$g\cdot\gamma=\Gr(\theta)(g,\gamma)=\theta\circ (g\times \gamma)\circ i.$$
The part (iii) is straightforward by the construction of $G$-action.
\end{proof}

Assume $\mathfrak X$ is a quasi-compact stft formal $R$-scheme. By \cite{Se}, the image $\pi_n(\Gr(\mathfrak X))$ of $\Gr(\mathfrak X)$ is a constructible subset of $\Gr_n(X_n)$, and it is also $G$-invariant by construction. If, in addition, $\mathfrak X$ is smooth of pure relative dimension $d$, we can extend \cite[Lem. 3.4.2]{Se} to the equivariant setting in which the $\mathfrak X_s$-morphism $\pi_n$ is $G$-equivariant surjective and the $\mathfrak X_s$-morphism $\pi_n^{n+1}$ is a $G$-equivariant locally trivial fibration in the Zariski topology with fiber $\mathbb A_{\mathfrak X_s}^d$ for any $n\in\mathbb N$. Furthermore, by using \cite[Lem. 4.3.25]{Se} and the construction of $G$-action we get

\begin{proposition}\label{G-transition}
Let $\mathfrak X$ be a flat quasi-compact stft formal $R$-scheme of relative dimension $d$ which is endowed with good $G$-action. Put
$$\Gr^{(e)}(\mathfrak X)=\Gr(\mathfrak X)\setminus\pi_e^{-1}(\Gr_e((X_e)_{\sing})),$$
where $(X_e)_{\sing}$ is the closed subscheme of non-smooth points of $X_e$. Then, there is an integer $c\geq 1$ such that, for any $e$ and $n$ in $\mathbb N$ with $n\geq ce$, the $\mathfrak X_s$-morphism projection 
$$\pi_{n+1}(\Gr(\mathfrak X))\to \pi_n(\Gr(\mathfrak X))$$ 
is a $G$-equivariant piecewise trivial fibration over $\pi_n(\Gr^{(e)}(\mathfrak X))$ with fiber $\mathbb A_{\mathfrak X_s}^d$. 
\end{proposition}

Recall that a subset $\mathscr A$ of $\Gr(\mathfrak X)$ is {\it cylindrical of level} $n$ if $\mathscr A=\pi_n^{-1}(C)$ with $C$ a constructible subset of $\Gr_n(X_n)$. Assume that $\mathfrak X$ is a stft formal $R$-scheme endowed with good $G$-action. If $C$ is a $G$-invariant constructible subset of $\Gr_n(X_n)$, then $\mathscr A=\pi_n^{-1}(C)$ is a $G$-invariant cylinder of $\Gr(\mathfrak X)$. If $\mathfrak X$ is as in Proposition \ref{G-transition}, then a $G$-invariant cylinder $\mathscr A$ of $\Gr(\mathfrak X)$ is said to be {\it stable of level} $n$ if it is $G$-invariant cylindrical of level $n$, and for every $n\leq m$ in $\mathbb N$, the $\mathfrak X_s$-morphism $\pi_m(\Gr(\mathfrak X))\to \pi_n(\Gr(\mathfrak X))$ is a $G$-equivariant piecewise trivial fibration over $\pi_m(\mathscr A)$ with fiber $\mathbb A_{\mathfrak X_s}^{(m-n)d}$. A $G$-invariant cylinder $\mathscr A$ is {\it stable} if it is stable of some level. Let $\mathbf C_{\mathfrak X}^G$ be the set of stable $G$-invariant cylinders of $\Gr(\mathfrak X)$, where $G$ is a finite group $k$-scheme.  

\begin{proposition}\label{measure}
Let $\mathfrak X$ be a flat quasi-compact stft formal $R$-scheme of relative dimension $d$ endowed with good $G$-action. Then there exists a unique additive mapping
$$\mu_{\mathfrak X}^G\colon \mathbf C_{\mathfrak X}^G\to \mathscr M_{\mathfrak X_s}^G$$ 
such that for any $G$-invariant cylinder $\mathscr A\subseteq \Gr(\mathfrak X)$ stable of level $n$,
$$\mu_{\mathfrak X}^G(\mathscr A)=[\pi_n(\mathscr A)\to \mathfrak X_s]\L^{-(n+1)d}$$ 
\end{proposition}

\begin{proof}
For $\mathscr A\subseteq \Gr(\mathfrak X)$ being a $G$-invariant cylinder stable of level $n$, it is a $G$-invariant cylinder stable of any level $m\geq n$. Applying Theorem \ref{keycoro} we have 
$$[\pi_m(\mathscr A)\to \mathfrak X_s]=[\pi_n(\mathscr A)\to \mathfrak X_s] \L^{(m-n)d}$$ 
in $\mathscr M_{\mathfrak X_s}^G$. It implies that $[\pi_m(\mathscr A)\to \mathfrak X_s] \L^{-(m+1)d}$ is independent of $m\geq n$, and we define 
$$\mu_{\mathfrak X}^G(\mathscr A):=[\pi_n(\mathscr A)\to \mathfrak X_s] \L^{-(n+1)d}.$$	
Forgetting the action it is shown in \cite[Prop. 4.3.13]{Se} the additivity of $\mu_{\mathfrak X}^G$. For such an $\mathscr A$, if there exist $G$-invariant cylinders $\mathscr A'$ and $\mathscr A''$ stable of level $\geq n$ such that $\mathscr A=\mathscr A'\cup \mathscr A''$, then $\mathscr A'\cap \mathscr A''$ is also a $G$-invariant stable cylinder of level $\geq n$. Therefore, $\mu_{\mathfrak X}^G$ is also additive in the $G$-action setting.	
\end{proof}

\begin{definition}\label{def-stft}
Let $G$ be a finite group $k$-scheme. For any $\mathscr A$ in $\mathbf C_{\mathfrak X}^G$ and any function $\alpha\colon \mathscr A\to \mathbb Z \cup \{\infty\}$, we say that $\L^{-\alpha}$ is {\it naively $G$-integrable}, or that $\alpha$ is {\it naively exponentially $G$-integrable}, if $\alpha$ takes only finitely many values in $\mathbb Z$ and if all the fibers of $\alpha$ are in $\mathbf C_{\mathfrak X}^G$ (such a function $\alpha$ is said to be {\it simple}). We define the {\it motivic $G$-integral} of $\L^{-\alpha}$ on $\mathscr A$ as follows
$$ \int_{\mathscr A}\L^{-\alpha}d\mu_{\mathfrak X}^G:=\sum_{n\in\mathbb Z}\mu_{\mathfrak X}^G(\alpha^{-1}(n))\L^{-n}\ \in \mathscr M_{\mathfrak X_s}^G.$$
\end{definition}

\begin{remark}
When considering the version without $G$-action, i.e., $\mathfrak X$ is endowed with trivial $G$-action, we shall write simply $\mu$ (resp. $\int_{\mathscr A}\L^{-\alpha} d\mu$) in stead of $\mu_{\mathfrak X}^G$ (resp. $\int_{\mathscr A}\L^{-\alpha} d\mu_{\mathfrak X}^G$).	
\end{remark}

\subsection{Change of variables formula}
Let $\mathfrak X$ be a quasi-compact generically smooth flat stft formal $R$-schemes. By \cite[Thm. 8.0.4]{Se}, there exist a smooth quasi-compact stft formal $R$-schemes $\mathfrak Y$ and an $R$-morphism $\mathfrak h\colon \mathfrak Y \to \mathfrak X$ such that $\Gr(\mathfrak Y)\to \Gr(\mathfrak X)$ is a bijection. Let $\mathfrak X$ and $\mathfrak Y$ be endowed with good $G$-actions, and $\mathfrak h\colon \mathfrak Y \to \mathfrak X$ be $G$-equivariant, where $G$ is a finite group $k$-scheme. By Proposition \ref{action-on-Greenberg}, the induced morphisms $\Gr(\mathfrak h)$ and $\Gr_n(h_n)$ are $G$-equivariant. In Theorem \ref{change-of-variable} we use the function $\ord_{\varpi}(\det\Jac_{\mathfrak h})\colon \Gr(\mathfrak X)\setminus \Gr(\mathfrak X_{\sing}) \to \mathbb N\cup \{\infty\}$ defined in \cite[Sect. 4]{NS1}. Here, we denote by $\mathfrak X_{\sing}$ the closed formal subscheme of $\mathfrak X$ on which the structural morphism $\mathfrak X\to \Spf R$ is non-smooth. By abuse of notation, the symbol $\L$ will stand for both $\L_{\mathfrak X_s}$ and $\L_{\mathfrak Y_s}$. 

\begin{theorem}\label{change-of-variable}
Let $\mathfrak X$ and $\mathfrak Y$ be quasi-compact flat stft formal $R$-schemes endowed with good $G$-actions, purely of the same relative dimension. Assume that $\mathfrak X$ is generically smooth and $\mathfrak Y$ is smooth over $R$. Let $\mathfrak h\colon \mathfrak Y \to \mathfrak X$ be a $G$-equivariant morphism of formal $R$-schemes such that $\mathfrak h_{\eta}$ is \'etale and $\mathfrak Y_{\eta}(K^{sh})=\mathfrak X_{\eta}(K^{sh})$. If $\alpha$ is a naively exponentially $G$-integrable function on $\Gr(\mathfrak X)\setminus \Gr(\mathfrak X_{\sing})$, so is $\alpha\circ \Gr(\mathfrak h)+\ord_{\varpi}(\det\Jac_{\mathfrak h})$ on $\Gr(\mathfrak Y)$. Moreover, in $\mathscr M_{\mathfrak X_s}^G$,
$$\int_{\Gr(\mathfrak X)\setminus \Gr(\mathfrak X_{\sing})}\L^{-\alpha}d\mu_{\mathfrak X}^G=(\mathfrak h_s)_!\int_{\Gr(\mathfrak Y)}\L^{-\alpha\circ \Gr(\mathfrak h)-\ord_{\varpi}(\Jac_{\mathfrak h})}d\mu_{\mathfrak Y}^G.$$
\end{theorem} 


\begin{proof}
The proof is essentially the same as that of  \cite[Thm. 8.0.5]{Se} or \cite[Thm. 4.3.1]{CNS}, adapted to the equivariant setting. The reader can find a detailed proof at arXiv:2206.01005v1.

\end{proof}


\subsection{Motivic $G$-integral on stft formal schemes of gauge forms}
Let $\mathfrak X$ be a flat stft generically smooth formal $R$-scheme of pure relative dimension $d$, and let $\widetilde\omega$ be a differential form in $\Omega_{\mathfrak X/R}^d(\mathfrak X)$. Let $x$ be in $\Gr(\mathfrak X)\setminus \Gr(\mathfrak X_{\sing})$ which is defined over some field extension $k'$ of $k$. Put $R'=R\widehat{\otimes}_kk'$ and consider the morphism of formal $R$-schemes $\gamma\colon \Spf R' \to \mathfrak X$ corresponding to $x$. Since $\big(\gamma^*\Omega_{\mathfrak X/R}^d\big)/(\text{torsion})$ is a free $\mathcal O_{R'}$-module of rank one, we have either $\gamma^*\widetilde\omega=0$ or $\gamma^*\widetilde\omega=\alpha \varpi^n$ for some nonzero $\alpha\in \mathcal O_{R'}$ and $n\in\mathbb N$. We define $\ord_{\varpi}(\widetilde\omega)(x)=n$ if $\gamma^*\widetilde\omega=\alpha \varpi^n$, and $\ord_{\varpi}(\widetilde\omega)(x)=\infty$ if $\gamma^*\widetilde\omega=0$. Consider the canonical isomorphism shown in \cite[Prop. 1.5]{BLR95} as follows
\begin{align}\label{isodiff}
\Omega_{\mathfrak X/R}^d(\mathfrak X)\otimes_RK \cong \Omega_{\mathfrak X_{\eta}/K}^d(\mathfrak X_{\eta}).
\end{align}

\begin{definition}
A {\it gauge} form $\omega$ on $\mathfrak X_{\eta}$ is a global section of the differential sheaf $\Omega_{\mathfrak X_{\eta}/K}^d$ such that it generates the sheaf at every point of $\mathfrak X_{\eta}$. 
\end{definition}

For any gauge form $\omega$ on $\mathfrak X_{\eta}$, there exist $\widetilde\omega\in \Omega_{\mathfrak X/R}^d(\mathfrak X)$ and $n\in \mathbb N$ such that $\omega=\varpi^{-n}\widetilde{\omega}$, hence we put $\ord_{\varpi,\mathfrak X}(\omega):=\ord_{\varpi}(\widetilde{\omega})-n$, which defines a $\mathbb Z$-value function 
$$\ord_{\varpi,\mathfrak X}(\omega)\colon \Gr(\mathfrak X)\setminus \Gr(\mathfrak X_{\sing})\to\mathbb Z\cup \{\infty\}.$$ 
This definition is independent of the choice of $\widetilde{\omega}$ thanks to \cite[Sect. 4.1]{LS}. 

\begin{lemma}\label{orderomega}
Let $\mathfrak X$ be a flat stft generically smooth formal $R$-scheme of pure relative dimension $d$. Assume that $\mathfrak X$ is endowed with a good $G$-action $\theta\colon G\times_k\mathfrak X\to \mathfrak X$, where $G$ is a finite group $k$-scheme. If $\omega$ is in $\Omega_{\mathfrak X_{\eta}/K}^d(\mathfrak X_{\eta})$ is a gauge form, the function $\ord_{\varpi,\mathfrak X}(\omega)$ is naively exponentially $G$-integrable.
\end{lemma}

\begin{proof}
The first part is proved as in \cite[Thm.-Def. 4.1.2]{LS}. Assume that there is an open dense smooth formal subscheme $\mathfrak Y$ of $\mathfrak X$ such that $\mathfrak Y_{\eta}$ is an open rigid subspace of $\mathfrak X_{\eta}$ and $\mathfrak Y(R^{sh})\to \mathfrak X_{\eta}(K^{sh})$ is a bijection, i.e. $\mathfrak Y$ is a weak N\'eron model of $\mathfrak X_{\eta}$. The module $\Omega_{\mathfrak Y/R}^d$ is locally free of rank one over $\mathcal O_{\mathfrak Y}$ because of the smoothness of $\mathfrak Y$, so there is a open covering $\{\mathfrak U_i\}$ of $\mathfrak Y$ such that $\Omega_{\mathfrak Y/R}^d(\mathfrak U_i)$ is free of rank one. Hence, for every $i$, there is an $f_i$ in $\mathcal O_{\mathfrak Y}(\mathfrak U_i)$ such that 
\begin{align*}
\widetilde{\omega}\mathcal O_{\mathfrak Y}(\mathfrak U_i)\otimes(\Omega_{\mathfrak Y/R}^d(\mathfrak Y))^{-1} \cong (f_i)\mathcal O_{\mathfrak Y}(\mathfrak U_i),
\end{align*}
in which we use the expression $\omega=\varpi^{-n}\widetilde{\omega}$. It implies that the restriction of the function $\ord_{\varpi}(\widetilde{\omega})$ to $\mathfrak U_i$ is equal to $\ord_{\varpi}(f_i)$ which assigns $\ord_{\varpi}(f_i(u))$ to a point $u\in \Gr(\mathfrak U_i)$. Let $f$ be the global section of $\mathcal O_{\mathfrak Y}$ such that $f=f_i$ on $\mathfrak U_i$. By glueing $\ord_{\varpi}(f_i)$'s altogether we get a function 
$$\ord_{\varpi}(f)\colon \Gr(\mathfrak Y)\to\mathbb Z \cup \{\infty\},$$ 
whose restriction to $\Gr(\mathfrak X)\setminus \Gr(\mathfrak X_{\sing})\subseteq \Gr(\mathfrak X)=\Gr(\mathfrak Y)$ coincides with $\ord_{\varpi}(\widetilde{\omega})$. Since $f$ induces an invertible function on $\mathfrak X_{\eta}$, by the Maximum Modulus Principle (see \cite{BLR95}), $\ord_{\varpi}(f)$ has only finitely many values, and so does $\ord_{\varpi, \mathfrak X}(\omega)$.  

As in the proof of \cite[Thm.-Def. 4.1.2]{LS}, the fibers of $\ord_{\varpi}(\widetilde{\omega})=\ord_{\varpi}(f)$ are stable cylinders. As shown below, they are $G$-invariant. Let $x$ be a point of $\Gr(\mathfrak X)\setminus \Gr(\mathfrak X_{\sing})$, which we assume to have residue field $k'$. Put $R'=R\widehat{\otimes}_kk'$. Let $\gamma\colon \Spf R' \to \mathfrak X$ be the morphism corresponding to $x$. For every $g$ in $G$, we denote by $g':=g(k')$ the corresponding element in $G(k')$. Then the point $g\cdot x$ of $\Gr(\mathfrak X)$ corresponds to the morphism $\gamma':=g'\cdot\gamma=\theta \circ (g'\times \gamma)$ from $\Spf R'$ to $\mathfrak X$. Since $g'\in G(k')$, we can prove that $g\cdot x$ is also in $\Gr(\mathfrak X)\setminus \Gr(\mathfrak X_{\sing})$ and that $\gamma^*\widetilde\omega$, $\gamma'^*\widetilde\omega$ have the same order in $\varpi$. Hence two points $x$ and $g\cdot x$ belong to a same fiber of $\ord_{\varpi}(\widetilde{\omega})$, i.e. the function $\ord_{\varpi,\mathfrak X}(\omega)$ is naively exponentially $G$-integrable.
\end{proof}


\begin{definition}\label{int_stft}
Let $\mathfrak X$ be a flat stft generically smooth formal $R$-scheme of pure relative dimension $d$ endowed with good $G$-action. For any gauge form $\omega$ on $\mathfrak X_{\eta}$, the integral
$$\int_{\mathfrak X}|\omega|:=\int_{\Gr(\mathfrak X)\setminus \Gr(\mathfrak X_{\sing})}\L^{-\ord_{\varpi, \mathfrak X}(\omega)}d\mu_{\mathfrak X}^G\ \in \mathscr M_{\mathfrak X_s}^G$$
is called the {\it motivic $G$-integral of $\omega$ on $\mathfrak X$}, and the integral
$$\int_{\mathfrak X_{\eta}}|\omega|:=\int_{\mathfrak X_s}\int_{\mathfrak X}|\omega|\ \in \mathscr M_k^G$$
is called the {\it motivic $G$-integral of $\omega$ on $\mathfrak X_{\eta}$}.	
\end{definition}

\begin{remark}\label{remark1}
The proof of Lemma \ref{orderomega} still works for any differential form $\omega\in \Omega_{\mathfrak X_{\eta}/K}^d(\mathfrak X_{\eta})$, which is not necessarily a gauge form, provided that $\ord_{\varpi,\mathfrak X}(\omega)$ has only finitely many values. Then the integral $\int_{\mathfrak X}|\omega|$ in Definition \ref{int_stft} is also well defined for the form $\omega$. 	
\end{remark}

\begin{lemma}\label{stft-connected}
Let $\mathfrak X$ be a flat smooth stft formal $R$-scheme of pure relative dimension $d$ endowed with good $G$-action. Let $\omega$ be a gauge form on $\mathfrak X_{\eta}$. Assume $\mathfrak X$ has connected components $\mathfrak X_i$, $i\in I$ (with $I$ finite). Then, for $i\in I$, the function $\ord_{\varpi,\mathfrak X_i}(\omega):=\ord_{\varpi,\mathfrak X_i}(\omega|_{(\mathfrak X_i)_{\eta}})$ is constant on $\Gr(\mathfrak X_i)$. If $(\mathfrak X_i)_s$ is $G$-invariant for every $i\in I$, then
$$\int_{\mathfrak X}|\omega|=\L^{-d}\sum_{i\in I}\left[(\mathfrak X_i)_s\hookrightarrow \mathfrak X_s\right] \L^{-\ord_{\varpi,\mathfrak X_i}(\omega)}$$
in $\mathscr M_{\mathfrak X_s}^G$. Denote by $\ord_C(\omega)$ the constant $\ord_{\varpi,\mathfrak X_i}(\omega)$ if $C$ is a connected component of $(\mathfrak X_i)_s$. Let $\mathcal C(\mathfrak X_s)$ be a family of constructible subsets of $\mathfrak X_s$ such that
\begin{itemize}
	\item $\mathfrak X_s=\bigcup_{C\in \mathcal C(\mathfrak X_s)}C$, in which $C\cap C'=\emptyset$ if $C, C'\in \mathcal C(\mathfrak X_s)$ and $C\not= C'$;
	
	\item every $C\in \mathcal C(\mathfrak X_s)$ is a union of connected components $C_1 ,\dots, C_r$ of $\mathfrak X_s$ such that $\ord_{C_1}(\omega)=\cdots=\ord_{C_r}(\omega)$ (denote this number by $\ord_C(\omega)$) and $C$ is $G$-invariant.
\end{itemize}
Then the following holds in $\mathscr M_{\mathfrak X_s}^G$:
$$\int_{\mathfrak X}|\omega|=\L^{-d}\sum_{C\in \mathcal C(\mathfrak X_s)}\left[ C\hookrightarrow \mathfrak X_s\right] \L^{-\ord_C(\omega)}.$$	
\end{lemma}

\begin{proof}
The first statement is the same as in \cite[Lem. 6.4.]{NS1} (see also \cite[Prop. 4.3.1]{LS}) using the hypothesis that $\omega$ is a gauge form on $\mathfrak X_{\eta}$. The rest completely follows from the definition of motivic $G$-integral of a gauge form (cf. Def. \ref{def-stft}, \ref{int_stft}).
\end{proof}

\section{Equivariant motivic integration on special formal schemes}
In this section, $R$ is still a complete discrete valuation ring of equal characteristic, while $K$, $k$ and $\varpi$ respectively denote the fraction field, the residue field and the fixed uniformizing parameter of $R$.

\subsection{Special formal schemes}
Let $R\{x_1,\dots, x_m\}[[y_1,\dots,y_{m'}]]$ be the mixed formal power series $R$-algebra which is the $R\{x_1,\dots, x_m\}$-algebra of formal power series in $y_1,\dots,y_{m'}$. We can prove that $R\{x_1,\dots, x_m\}[[y_1,\dots,y_{m'}]]=R[[y_1,\dots,y_{m'}]]\{x_1,\dots, x_m\}$. A topological $R$-algebra $A$ is called {\it special} if $A$ is a Noetherian adic ring and the $R$-algebra $A/J$ is finitely generated for some ideal of definition $J$ of $A$. By \cite{Ber96}, we have an equivalent definition, which states that a topological $R$-algebra $A$ is special if and only if $A$ is topologically $R$-isomorphic to a quotient the $R$-algebra $R\{x_1,\dots, x_m\}[[y_1,\dots,y_{m'}]]$ for some $m, m'\in\mathbb N^*$.  

\begin{definition}
A {\it special} formal scheme is a separated Noetherian adic formal scheme $\mathfrak X$ that is a finite union of open affine formal schemes of the form $\Spf A$ with $A$ a Noetherian special $R$-algebra. A morphism $\mathfrak{Y}\to\mathfrak{X}$ of special formal $R$-schemes is said to be {\it locally of finite type} if locally it is isomorphic to a morphism of the form $\Spf B\to \Spf A$, where $A\to B$ is a morphism of finite type of Noetherian special $R$-algebras. A morphism $\mathfrak{Y}\to\mathfrak{X}$ of special formal $R$-schemes is {\it (locally) adic} if locally it is isomorphic to a morphism of the form $\Spf B\to \Spf A$, where $A\to B$ is an adic morphism, i.e. there is an ideal of definition $J$ of $A$ such that the topology on $B$ is $J$-adic.
\end{definition}

The category of special formal $R$-schemes admits fiber products, and it contains the category of formal $R$-schemes topologically of finite type as a full subcategory. If $\mathfrak X$ is a special formal $R$-scheme, any formal completion of $\mathfrak X$ is also a special formal $R$-scheme. 


In the category of stft formal $R$-schemes every morphism is automatically an adic morphism, but in general this is not true in the category of special formal $R$-schemes. In the rest of this article, we shall always consider {\it adic} morphisms between special formal $R$-schemes $\mathfrak f\colon \mathfrak Y\to \mathfrak X$, because it allows to induce $k$-morphisms at the reduction level $\mathfrak f_0\colon \mathfrak Y_0\to \mathfrak X_0$. For short, from now on, saying morphisms of special formal $R$-schemes we means adic morphisms.

As explained in \cite[0.2.6]{Berthelot} and \cite[Sect. 2.1]{Ni2}, one first considers the affine case $\mathfrak X=\Spf A$, where $A$ is a special $R$-algebra. Denote by $J$ the largest ideal of definition of $A$ and consider for each $n\in \mathbb N^*$ the subalgebra $A\left[\varpi^{-1}J^n\right]$ of $A\otimes_RK$ generated by $A$ and $\varpi^{-1}J^n$. Let $B_n$ be the $J$-adic completion of $A[\varpi^{-1}J^n]$. Then we have the affinoid $K$-algebra $C_n:=B_n\otimes_RK$. The inclusion $J^{n+1}\subseteq J^n$ gives rise naturally to a morphism of affinoid $K$-algebras $C_{n+1}\to  C_n$, which in its turn induces an open embedding of affinoid $K$-spaces $\Spm(C_n)\to \Spm(C_{n+1})$. Then the {\it generic fiber} $\mathfrak X_{\eta}$ of $\mathfrak X$ is defined to be $\varinjlim_{n}\Spm(C_n)=\bigcup_{n\in \mathbb N^*}\Spm(C_n)$. The {\it generic fiber} $\mathfrak X_{\eta}$ of an arbitrary special formal $R$-scheme $\mathfrak X$ can be obtained by a glueing procedure, i.e. one covers $\mathfrak X$ by open affine formal $R$-subscheme $\mathfrak X_i$, $i\in I$, and glues $\mathfrak X_{i\eta}$ into $\mathfrak X_{\eta}$ due to the method introduced in \cite[Prop. 0.2.3]{Berthelot}. In general, $\mathfrak X_{\eta}$ is a rigid $K$-variety, which is separated but not necessarily quasi-compact (cf. \cite{Berthelot}). The correspondence $\mathfrak X\mapsto \mathfrak X_{\eta}$ is a functor from the category of special formal $R$-schemes to the category of separated rigid $K$-varieties. 

Note that the reduction $\mathfrak X_0$ of any Noetherian $\varpi$-adic formal scheme is already mentioned in Section \ref{section3.1}. The {\it specialization map} $\sp\colon \Spf A_{\eta}\to \Spf A$ for the case of affine special formal $R$-schemes is defined as follows. Let $x$ be in $\Spf A_{\eta}$, let $I\subseteq A\otimes_RK$ be the maximal ideal in $A\otimes_RK$ corresponding to $x$, and let $I'=I\cap A$. By construction, $\sp(x)$ is the unique maximal ideal of $A$ containing $\varpi$ and $I'$. If $Z$ is a locally closed subscheme of $(\Spf A)_0$, $\sp^{-1}(Z)$ is an open rigid $K$-subvariety of $(\Spf A)_{\eta}$, which is canonically isomorphic to the generic fiber of the formal completion of $\Spf A$ along $Z$ (cf. \cite[Sect. 7.1]{deJ}). In general, the construction of the specialization map can be generalized to any special formal $R$-scheme $\mathfrak X$ (see \cite{deJ}).

Let $\mathfrak X$ be a special formal scheme of pure relative dimension $d$. Thanks to \cite[Sect. 7]{deJ} and \cite[Sect. 2.1]{Ni2}, we have a natural injective map $\Phi\colon \Omega_{\mathfrak X/R}^d(\mathfrak X)\otimes_RK\to \Omega_{\mathfrak X_{\eta}/K}^d(\mathfrak X_{\eta})$, which factors uniquely through the sheafification map $\Omega_{\mathfrak X/R}^d(\mathfrak X)\otimes_RK \to \big(\Omega_{\mathfrak X/R}^d\otimes_RK\big)(\mathfrak X)$, namely,
\begin{displaymath}
\xymatrix{
	\Omega_{\mathfrak X/R}^d(\mathfrak X)\otimes_RK\ \ar@{^{(}->}[rr]^{\Phi}\ar@{->}[dr]_{\Pi}&&\ \Omega_{\mathfrak X_{\eta}/K}^d(\mathfrak X_{\eta})\\
	&\big(\Omega_{\mathfrak X/R}^d\otimes_RK\big)(\mathfrak X)\ar@{->}[ur]_{\Psi}&
}
\end{displaymath}
If $\mathfrak X$ is a stft formal $R$-scheme, $\Phi$ is an isomorphism (see Eq. (\ref{isodiff})); also, if $\mathfrak X$ is an affine special formal $R$-scheme, the homomorphism $\Pi$ is an isomorphism. In the general case, we always have $\mathrm{Im}(\Phi)=\mathrm{Im}(\Psi)$. A gauge form $\omega$ on $\mathfrak X_{\eta}$ which lies in $\mathrm{Im}(\Phi)$ is called {\it $\mathfrak X$-bounded}. 
Clearly, if $\mathfrak X$ is a stft formal $R$-scheme, every gauge form on $\mathfrak X_{\eta}$ is an $\mathfrak X$-bounded gauge form; but this does not hold in general. 


\begin{definition}
Let $\mathfrak X$ be a special formal $R$-scheme. An {\it adic} $G$-action on $\mathfrak X$ is a $G$-action $\theta\colon G\times_k\mathfrak X \to \mathfrak X$ such that regarding as a morphism over $\Spf R$, $\theta$ is adic.

Let $\mathfrak X$ and $\mathfrak Y$ be special formal $R$-schemes each of which is endowed with an adic $G$-action. An adic morphism of special formal schemes $\mathfrak f \colon \mathfrak Y\to \mathfrak X$ is called {\it adic $G$-equivariant} if it is compatible with the adic actions $G\times_k\mathfrak X \to \mathfrak X$ and $G\times_k\mathfrak Y \to \mathfrak Y$.
\end{definition}  

Clearly, if $\theta \colon G\times_k\mathfrak X \to \mathfrak X$ is an adic $G$-action on $\mathfrak X$, it induces naturally a $G$-action on $\mathfrak X_0$ which is the $k$-morphism $\theta_0\colon G\times_k\mathfrak X_0 \to \mathfrak X_0$.

\subsection{Motivic $G$-integral on special formal schemes}
Let $\mathfrak X$ be a Noetherian adic formal $R$-scheme, let $\mathcal J$ be the largest ideal of definition of $\mathfrak X$, and $\mathcal I$ a coherent ideal sheaf on $\mathfrak X$. The {\it formal blowup of $\mathfrak X$ with center $\mathcal I$} is the morphism of formal schemes
\begin{align}\label{blowup}
\pi\colon \varinjlim_{n\in \mathbb N^*} \Proj\Big(\bigoplus_{m\geq 0}\mathcal I^m\otimes_{\mathcal O_{\mathfrak X}}(\mathcal O_{\mathfrak X}/\mathcal J^n)\Big)\to\mathfrak X.
\end{align}
The formal blowup of $\mathfrak X$ with center $\mathcal I$ is an adic formal $R$-scheme on which the ideal generated by $\mathcal I$ is invertible. The blowup has the universality, it also commutes with flat base change, with the completion of $\mathfrak X$ along a closed $k$-subscheme of $\mathfrak X_s$ (cf. \cite[Prop. 2.16]{Ni2}). If $\mathfrak X$ is a special formal $R$-scheme and $\mathcal I$ is open with respect to the $\varpi$-adic topology, then the blowup $\pi$ is called {\it admissible}. By \cite[Cor. 2.17]{Ni2}, if $\mathfrak Y\to\mathfrak X$ is an admissible blowup, then $\mathfrak Y$ is a special formal $R$-scheme, and if, in addition, $\mathfrak X$ is $R$-flat, so is $\mathfrak Y$. Furthermore, the induced morphism of rigid $K$-varieties $\mathfrak Y_{\eta}\to\mathfrak X_{\eta}$ is an isomorphism due to \cite[Prop. 2.19]{Ni2}. 

\begin{definition}
Let $\mathfrak X$ be a flat special formal $R$-scheme, and $\mathcal I$ a coherent ideal sheaf on $\mathfrak X$ which contains $\varpi$. Let $\pi\colon \mathfrak Y\to \mathfrak X$ be the admissible blowup with center $\mathcal I$. If $\mathfrak U$ is the open formal subscheme of $\mathfrak Y$ where $\mathcal I\mathcal O_{\mathfrak Y}$ is generated by $\varpi$, the restriction $\pi\colon \mathfrak U\to\mathfrak X$ is called {\it the dilatation of $\mathfrak X$ with center $\mathcal I$}. 
\end{definition}

The dilatation of a special formal $R$-scheme $\mathfrak X$ always exists and it is a flat special formal $R$-scheme. Like admissible blowups, dilatations have the universality (cf. \cite[Prop. 2.22]{Ni2}), they commute with the formal completion along closed subschemes (cf. \cite[Prop. 2.21, 2.23]{Ni2}). 

\begin{proposition}[Nicaise \cite{Ni2}]\label{universal}
Let $\mathfrak X$ be a flat special formal $R$-scheme, and let $\mathfrak U\to\mathfrak X$ be the dilatation of $\mathfrak X$ with center $\mathcal I$ containing $\varpi$. If $\mathfrak X' \to \mathfrak X$ is a morphism of flat special formal $R$-schemes such that the induced morphism $\mathfrak X'_s \to \mathfrak X_s$ factors through the closed formal subscheme of $\mathfrak X_s$ defined by $\mathcal I$, there exists a unique morphism of formal $R$-schemes $\mathfrak X' \to \mathfrak U$ that makes the diagram 
\begin{displaymath}
\xymatrix{ &  \mathfrak U \ar[d]^{} \ar@{<--}[dl]_{} \\  \mathfrak X' \ar[r]^{} &  \ \mathfrak X.}
\end{displaymath}
commute. Furthermore, if $\mathcal I$ is open, then $\mathfrak U$ is a stft formal $R$-scheme.
\end{proposition}

Let $G$ be a finite group $k$-scheme. Due to Definition \ref{action-k}, a $G$-action on $\mathfrak X$ is a $G$-action on $\mathfrak X$ when viewed as a formal $k$-scheme such that the structural morphism is $G$-equivariant. 

\begin{proposition}\label{G-dilatation}
Let $\mathfrak X$ be a flat special formal $R$-scheme endowed with a good adic $G$-action $\theta$. Let $\pi\colon \mathfrak U\to\mathfrak X$ be the dilatation of $\mathfrak X$ with center $\mathcal I$ containing $\varpi$, and let $\mathfrak Z$ be the closed formal subscheme of $\mathfrak X_s$ defined by $\mathcal I$. Assume that $\mathfrak Z$ is $G$-invariant. Then there exists a good adic $G$-action $\pi^*\theta$ on $\mathfrak U$ such that the dilatation $\pi$ is $G$-equivariant. The formal scheme $\mathfrak U$ together with this action is called the {\bf $G$-dilatation} of $\mathfrak X$
\end{proposition}

\begin{proof}
Consider the morphism of formal $k$-schemes $\pi'=\theta\circ (\Id\times\pi)\colon G\times_k\mathfrak U\to  G\times_k\mathfrak X \to  \mathfrak X$. Since $\pi$ is the dilatation with center $\mathcal I$, and since $\mathfrak Z$ is $G$-invariant, it follows that the induced morphism
$\pi'_s\colon  G\times_k\mathfrak U_s\to  G\times_k\mathfrak X_s \to  \mathfrak X_s$ factors through $\mathfrak Z$. Note that $\mathfrak X$ is a formal $R$-scheme and therefore $\pi'$ can be regarded as a morphism of  formal $R$-schemes. Applying Proposition \ref{universal} to $\pi'$ we obtain a unique morphism $\pi^*\theta\colon G\times_k\mathfrak U\to \mathfrak U$ making the diagram
\begin{equation*}
\begin{CD}
G\times_k\mathfrak U @>\pi^*\theta>> \mathfrak U\\
@V\Id\times\pi VV @VV\pi V\\
G\times_k\mathfrak X @>\theta>> \mathfrak X
\end{CD}\end{equation*}
commutes. Similarly, applying Proposition \ref{universal} we can show that $\pi^*\theta$ defines a good adic action of $G$ on $\mathfrak U$, hence the dilatation $\pi$ is $G$-equivariant due to the commutative diagram.
\end{proof}

The below definition is a generalization of \cite[Def. 4.1]{Ni2} to the $G$-equivariant setting.

\begin{definition}
A {\it $G$-N\'eron smoothening} for a special formal $R$-scheme $\mathfrak X$ is a $G$-equivariant morphism of special formal $R$-schemes $\mathfrak Y \to\mathfrak X$, with $\mathfrak Y$ adic smooth over $R$, such that the induced morphism $\mathfrak Y_{\eta} \to \mathfrak X_{\eta}$ is an open embedding satisfying $\mathfrak Y_{\eta}(K^{sh})=\mathfrak X_{\eta}(K^{sh})$, and the induced morphism $\mathfrak Y_{s} \to \mathfrak X_{s}$ factors through $\mathfrak X_{0}$.
\end{definition}

\begin{proposition}\label{G-smoothening}
Let $G$ be a smooth finite group $k$-scheme. Then every flat generically smooth special formal $R$-scheme $\mathfrak X$ endowed with a good adic $G$-action admits a $G$-N\'eron smoothening $\mathfrak h\colon \mathfrak Y \to\mathfrak X$.
\end{proposition}

\begin{proof}
First, we mention the following simple but useful observations: 
\begin{itemize}
\item[(i)] Every adic action $\theta$ of $G$ on $\mathfrak X$ is smooth as $G$ is smooth. This follows from the below commutative diagram where the isomorphism is defined by $(g,x)\mapsto (g,g^{-1}x)$
\begin{displaymath}
\xymatrix{
G\times_k\mathfrak X \ar[rr]^{\cong}\ar@{->}[dr]_{\theta}&&G\times_k\mathfrak X\ar@{->}[dl]^{\pr_2}\\
&\mathfrak X&
}
\end{displaymath}

\item[(ii)] The $R$-smooth locus $Sm(\mathfrak X)$ of $\mathfrak X$ is $G$-invariant. Since the structural morphism $\mathfrak f$ is $G$-equivariant, we have the following commutative diagram
\begin{displaymath}
\xymatrix@=2.5em{
	G\times_k\mathfrak X \ar[r]^{\theta}\ar[dr]^{\widetilde{\mathfrak f}}\ar[d]_{\Id\times \mathfrak f}& \mathfrak X \ar[d]^{\mathfrak f}\\
	G\times_k\Spf R\ar[r]^{\rho}& \Spf R
}
\end{displaymath}
\end{itemize}
By (i), the morphisms $\rho$ is smooth, so is the morphism $\widetilde{\mathfrak f}|_{G\times_k Sm(\mathfrak X)}$. Combining with the smoothness of $\theta$ (by (i))  we can deduce that for all $g\in G$ and all $x\in Sm(\mathfrak X)$, $\mathfrak f$ is smooth at $gx$ (\cite[Tag 02K5]{Stack}). This means that $Sm(\mathfrak X)$ is $G$-invariant.
	
Let $\pi\colon \mathfrak U\to\mathfrak X$ be the $G$-dilatation of $\mathfrak X$ with center $\mathfrak X_0$. Then $\mathfrak U_{\eta}(K^{sh})=\mathfrak X_{\eta}(K^{sh})$ by the universal property of the dilatation (cf. \cite[Lem. 4.3]{Ni2}). Applying \cite[3.4/2]{BLR90} and \cite[Thm. 3.1]{BS95} for $E:=\mathfrak U(R^{sh})$ we obtain a morphism $\mathfrak U'\to\mathfrak U$ which consists of a finite sequence of ($E$-permissible) blowups with centers contained in the non-smooth parts of the corresponding formal $R$-schemes, such that the $R$-smooth locus $Sm(\mathfrak U')$ is a N\'eron smoothening of $\mathfrak U$. Let us consider the finite sequence of ($E$-permissible) blowups $\mathfrak h \colon \mathfrak U'\to\mathfrak U$ constructed in \cite[3.4/2]{BLR90}. We first prove that the centers of these blowups are $G$-invariants. By induction, it suffices to prove that the center $Z$ in $\mathfrak U_s$ of the first blowup $\mathfrak h_1\colon\mathfrak U_1\to\mathfrak U$ is $G$-equivariant, since then, by the universal property of blowups, $\mathfrak U_1$ admits an adic action of $G$ such that $\mathfrak h_1$ is $G$-equivariant (see, Prop. \ref{G-dilatation}). The center $Z=Y_{\ell}$ of the first blowup $\mathfrak h_1$ is defined in \cite[3.4/2]{BLR90} as follows. Let $(\cdot)_k$ be the composition 
$$(\cdot)_k\colon \mathfrak U(R^{sh})\to \mathfrak U(\overline{k}) \xhookrightarrow{} \mathfrak U\times_R k=\mathfrak U_s,$$
where the first map is induced by the specilization $R^{sh}\to \overline{k}$. Note that, the adic action of $G$ on $\mathfrak U$ induces actions of $G(R^{sh})$ on $\mathfrak U(R^{sh})$ and $G$ on $\mathfrak U_s$ respectively. Therefore, the map $(\cdot)_k$ is equivariant, i.e. the following diagram commutes
\begin{equation*}
\begin{CD}
G(R^{sh})\times\mathfrak U(R^{sh}) @>>> \mathfrak U(R^{sh})\\
@V(\cdot)_k VV @VV(\cdot)_kV\\
G\times_k\mathfrak U_s @>>> \  \mathfrak U_s.
\end{CD}\end{equation*}
Set $F^1 := E$ and $Y_1 := \overline{F^1_k}$ the Zariski closure of $F^1_k:=(\cdot)_k(F^1)$ in $\mathfrak U_0$. Let $U_1$ be the largest open subscheme of $Y_1$ which is smooth over $k$ and where $\Omega^1_{\mathfrak X/R}|_{U^1}$ is locally free, and define 
$$E^1:=\left\{ a\in F^1 | a_k \in {U_1}\right\}=(\cdot)_k^{-1} (U_1).$$ 
Proceeding in the same way with $F^2 := F^1 \setminus E^1$, and so on, we obtain 
\begin{itemize}
\item[(a)] a decreasing sequence $F^1\supseteq F^2\supseteq \ldots$ in $E$, 
\item[(b)] subsets $E^1, E^2, \ldots$ such that $E = E^1\sqcup\cdots\sqcup  E^ i\sqcup F^{i+1}$ for all $i\geq 1$,
\item[(c)] dense open subschemes $U_i\subseteq  Y_i := \overline{F^i_k}$ such that $E^i_k\subseteq U_i$ and, moreover, $Y_{i+1}\subseteq  Y_i\setminus U_i$; in particular, $\dim Y_{i+1} < \dim Y_{i}$ if $Y_{i}\neq \emptyset$.
\end{itemize}
We see that $Y_{i+1}= \emptyset$ for all $i\geq\dim \mathfrak U_0$, denote by $\ell$ the smallest $i$ with this property, i.e. $Y_{\ell}\neq \emptyset$ and $Y_{\ell+1}= \emptyset$. By construction, $Z=Y_{\ell}$ is the center of the first blowup of the sequence $\mathfrak h \colon \mathfrak U'\to\mathfrak U$. Let us prove the $G$-invariance of $Z=Y_{\ell}$. It is obvious that $Y_1=\mathfrak U_0$ is $G$-invariant. Observe that $U_1=Sm(Y_1)\cap\iota^{-1}(Sm(\mathfrak U))$ is also $G$-invariant, where $\iota\colon\mathfrak U_s\xhookrightarrow{} \mathfrak U$ is the natural inclusion which is also $G$-equivariant. Since the map $(\cdot)_k$ is equivariant, the set $E^1$, and so $F^2$ are $G(R^{sh})$-invariant and $F^2_k$ is $G(k)$-invariant. Let $\pi\colon \mathfrak U_s\to \mathfrak U_s/G$ denote the geometric quotient morphism  \cite[Expos\'e V, Prop. 1.8]{Gro63}. We now claim that the geometric quotient morphism $\pi$ is open. Indeed, since $\pi$ is a geometric quotient, the topology in $\mathfrak U_s/G$ is the quotient topology. If $U$ is open in $\mathfrak U_s$, then $\pi^{-1}(\pi(U))=\bigcup_{g\in G} g\cdot U$, hence it is open in $\mathfrak U_s$. Therefore, $\pi(U)$ is open in $\mathfrak U_s/G$. Thus the Zariski closure $Y_2$ of $F^2_k$ in $\mathfrak U_s$ is the preimage of the Zariski closure of $\pi(F^2_k)$ in $\mathfrak U_s/G$, hence $ Y_2$  is $G$-invariant. By induction, we conclude that $Y_{\ell}$ is $G$-invariant. This proves our first assertion that the centers of the blowups occuring in $\mathfrak h \colon\mathfrak U'\to\mathfrak U$ are $G$-invariants. This implies, by the universal property of the blowup, there exists an adic action of $G$ on $\mathfrak U'$ such that $\mathfrak h$ is $G$-equivariant. Again, by our first observation $\mathfrak Y=Sm(\mathfrak U')$ is also $G$-invariant, hence $\mathfrak h\colon\mathfrak Y\to\mathfrak U$ is a $G$-N\'eron smoothening of $\mathfrak U$ and $\mathfrak X$.
\end{proof}

\begin{proposition}\label{dilatation_stft}
Let $G$ be a smooth finite group $k$-scheme. Let $\mathfrak X$ be a flat, generically smooth stft formal R-scheme endowed with good $G$-action, and let $U$ be a $G$-invariant closed subscheme of $\mathfrak X_0$. Let $\pi\colon \mathfrak U\to\mathfrak X$ be the $G$-dilatation of $\mathfrak X$ with center $U$. If $\omega$ is  a gauge form on $\mathfrak X_\eta$, the identity $\int_{\mathfrak U}|\omega|= \pi^*_s\int_{\mathfrak X}|\pi_{\eta}^*\omega|$ holds in $\mathscr M_{\mathfrak U_s}^G$.
\end{proposition}

\begin{proof}
Let $\mathfrak h\colon \mathfrak Y'\to\mathfrak X$ be the blowup of $\mathfrak X$ with center $U$. Using the same argument as in the proof of Proposition \ref{G-dilatation}, we can construct an action of $G$ on $\mathfrak Y'$ extending the action on $\mathfrak U$ such that $\mathfrak h$ is $G$-equivariant. Then the proof works on the same line as in \cite[Prop. 4.5]{Ni2} by using the $G$-N\'eron smoothening $\mathfrak g\colon \mathfrak Z'\to\mathfrak Y'$ as constructed in Proposition \ref{G-smoothening}.
\end{proof}

Remark from Proposition \ref{universal} that if $\pi\colon \mathfrak U\to\mathfrak X$ is the dilatation of a flat special formal $R$-scheme $\mathfrak X$ with center $\mathfrak X_0$, then $\mathfrak U$ is a flat stft formal $R$-scheme, and for any gauge form $\omega$ on $\mathfrak X_{\eta}$, the differential form $\pi_{\eta}^*\omega$ is also a gauge form on $\mathfrak U_{\eta}$.

\begin{definition}\label{spint}
Let $G$ be a finite group $k$-scheme. Let $\mathfrak X$ be a flat generically smooth special formal $R$-scheme endowed with a good adic $G$-action, and let $\pi\colon \mathfrak U\to\mathfrak X$ be the $G$-dilatation of $\mathfrak X$ with center $\mathfrak X_0$. For any gauge form $\omega$ on $\mathfrak X_{\eta}$, we define 
\begin{align*}
\int_{\mathfrak X}|\omega|:= {\pi_s}_!\int_{\mathfrak U}|\pi_{\eta}^*\omega|\quad \text{in} \ \ \mathscr M_{\mathfrak X_0}^G
\end{align*}
and call it the {\it motivic $G$-integral} of $\omega$ on $\mathfrak X$. The integral $\int_{\mathfrak X_{\eta}}|\omega|:=\int_{\mathfrak U_{\eta}}|\omega|$ in $\mathscr M_k^G$ is called the {\it motivic $G$-integral} of $\omega$ on $\mathfrak X_{\eta}$.
	
If $\mathfrak X$ is a generically smooth special formal $R$-scheme endowed with good adic $G$-action, we denote by $\mathfrak X^{\flat}$ its maximal flat closed subscheme (obtained by killing $\varpi$-torsion), and define the {\it motivic $G$-integral} of a gauge form $\omega$ on $\mathfrak X$ to be
\begin{align*}
\int_{\mathfrak X}|\omega|:= \int_{\mathfrak X^{\flat}}|\omega| \quad \text{in} \ \ \mathscr M_{\mathfrak X_0}^G.
\end{align*}	
In this case, the integral $\int_{\mathfrak X_{\eta}}|\omega|:=\int_{\mathfrak X_{\eta}^{\flat}}|\omega|$ in $\mathscr M_k^G$ is called the {\it motivic $G$-integral} of $\omega$ on $\mathfrak X_{\eta}$.
\end{definition}

\begin{remark}\label{remark2}
a) By Remark \ref{remark1}, $\int_{\mathfrak X}|\omega|$ can be defined for any differential form of maximal degree $\omega$ on $\mathfrak X_{\eta}$ provided $\ord_{\varpi,\mathfrak U}(\pi_{\eta}^*\omega)$ has only finitely many values in $\mathbb Z$.

b) In the stft case, $\int_{\mathfrak X}|\omega|$ is obtained from the integral defined in Definition \ref{int_stft} by the base change $\mathscr M_{\mathfrak X_s}^G\to \mathscr M_{\mathfrak X_0}^G$ (due to Proposition \ref{dilatation_stft}). Therefore, we shall use the same notation for these two integrals but mention the ring of integral values.
\end{remark}

\begin{proposition}\label{Int_G-smoothening}
Let $\mathfrak X$ be a generically smooth special formal $R$-scheme endowed with a good adic action of $G$ and let $\phi\colon \mathfrak Y \to\mathfrak X$ be a $G$-N\'eron smoothening. Then, for any gauge form $\omega$ on $\mathfrak X_{\eta}$, we have
$$\int_{\mathfrak X}|\omega|={\phi_s}_!\int_{\mathfrak Y}|\phi_{\eta}^*\omega|  \quad \text{in} \ \ \mathscr M_{\mathfrak X_0}^G.$$
\end{proposition}

\begin{proof}
Let $\pi\colon \mathfrak U\to\mathfrak X$ be the $G$-dilatation of $\mathfrak X$ with center $\mathfrak X_0$. Then therer exists a unique morphism of formal $R$-schemes $\psi\colon \mathfrak Y\to\mathfrak U$ by the universal property (cf. Proposition \ref{G-dilatation}) of $\pi$ such that $\phi=\pi\circ\psi$. Combining Proposition \ref{change-of-variable}, Definition \ref{int_stft} and \cite[Lem. 4.1.1]{LS} we obtain the identity
$\int_{\mathfrak U}|\omega| = {\psi_s}_!\int_{\mathfrak Y}|\omega|$, 
which holds in $\mathscr M_{\mathfrak U_s}^G$. Hence, 
\begin{align*}
\int_{\mathfrak X}|\omega|&= {\pi_s}_!\int_{\mathfrak U}|\omega| = {\pi_s}_!\Big({\psi_s}_!\int_{\mathfrak Y}|\omega|\Big) = {\phi_s}_!\int_{\mathfrak Y}|\omega|,
\end{align*}	
which holds in $\mathscr M_{\mathfrak X_0}^G$.
\end{proof}

\begin{theorem}[Special $G$-equivariant change of variables formula]\label{changeofvariables}
Let $G$ be a smooth finite group $k$-scheme. Let $\mathfrak X$ and $\mathfrak Y$ be generically smooth special formal $R$-schemes endowed with good adic actions of $G$, and let $\mathfrak h\colon \mathfrak Y \to\mathfrak X$ be an adic $G$-equivariant morphism of formal $R$-schemes such that  the induced morphism $\mathfrak Y_{\eta} \to \mathfrak X_{\eta}$ is an open embedding and $\mathfrak Y_{\eta}(K^{sh})=\mathfrak X_{\eta}(K^{sh})$. If  $\omega$ is a gauge form on $\mathfrak X_{\eta}$, then 
$$\int_{\mathfrak X}|\omega|={\mathfrak h_0}_!\int_{\mathfrak Y}| \mathfrak h_{\eta}^*\omega|\quad \text{in} \ \ \mathscr M_{\mathfrak X_0}^G.$$
\end{theorem}

\begin{proof}
Let $\phi\colon \mathfrak Z \to\mathfrak Y$ be a $G$-N\'eron smoothening of $\mathfrak Y$. Then $\mathfrak h\circ\phi\colon \mathfrak Z \to\mathfrak X$ is a $G$-N\'eron smoothening of $\mathfrak X$. It follows from Proposition \ref{Int_G-smoothening} that $\int_{\mathfrak Y}|\omega|={\phi_s}_!\int_{\mathfrak Z}|\omega|$ in $\mathscr M_{\mathfrak Y_0}^G$ and that $\int_{\mathfrak X}|\omega|= {(\mathfrak h\circ\phi)_s}_!\int_{\mathfrak Z}|\omega|$ in $\mathscr M_{\mathfrak X_0}^G$. Since $\mathfrak h$ is adic, $\mathfrak h_0$ is nothing but the restriction of $\mathfrak h_s$ on ${\mathfrak Y_0}$. Thus $\int_{\mathfrak X}|\omega|=  {\mathfrak h_0}_!\int_{\mathfrak Y}|\omega|$ in $\mathscr M_{\mathfrak X_0}^G$.
\end{proof}

\begin{proposition}\label{base_change}
Let $G$ be a smooth finite group $k$-scheme. Let $\mathfrak X$ be a generically smooth special formal $R$-scheme endowed with a good adic action of $G$, and let $U$ be a $G$-invariant locally closed subscheme of $\mathfrak X_0$. Denote by $\mathfrak U$ the formal completion of $\mathfrak X$ along $U$. Then for every gauge form $\omega$ on $\mathfrak X_{\eta}$, the integral $\int_{\mathfrak U}|\omega|$ is the image of $\int_{\mathfrak X}|\omega|$ under the base change
$$\mathscr M_{\mathfrak X_0}^G\to  \mathscr M_{ U}^G.$$
\end{proposition}

\begin{proof}
We prove only for the case that $U$ is a closed subscheme of $\mathfrak X_0$ since the proof for the case that $U$ is open is similar (and simpler). Let $\pi\colon \mathfrak X'\to\mathfrak X$ and $\pi'\colon \mathfrak U'\to\mathfrak U$ be the $G$-dilatations of  $\mathfrak X$ with center $\mathfrak X_0$ and of $\mathfrak U$ with center $U$, respectively. By \cite[Prop. 2.23]{Ni2}, there exists a  morphism  $\mathfrak i'\colon  \mathfrak U'\to\mathfrak X'$ such that the diagram
\begin{equation*}
\begin{CD}
\mathfrak U' @>\pi' >> \mathfrak U\\
@V \mathfrak i' VV @VV \mathfrak i V\\
\mathfrak X' @>\pi >> \mathfrak X
\end{CD}\end{equation*}  
commutes, and the morphism $\mathfrak i'\colon  \mathfrak U'\to\mathfrak X'$ is the dilatation of $\mathfrak X'$ with center $\mathfrak X'_s\times_{\mathfrak X_0}U$. Moreover, it follows from \cite[Prop. 2.21]{Ni2} that $\mathfrak U'$ is actually the formal completion of $\mathfrak X'$ along $\mathfrak X'_s\times_{\mathfrak X_0}U$, i.e. the following induced diagram is Cartesian
\begin{equation*}
\begin{CD}
\mathfrak U'_s @>\pi'_s >> \mathfrak U_s\\
@V \mathfrak i'_s VV @VV \mathfrak i_s V\\
\mathfrak X'_s @>\pi_s >> \mathfrak X_0
\end{CD}\end{equation*}  
Since $\pi_s$ and $\pi'_s$ factors through $ \mathfrak X_0$ and $U$ respectively, the diagram
\begin{equation*}
\begin{CD}
\mathfrak U'_s @>\pi'_s >>  U\\
@V \mathfrak i'_s VV @VV \mathfrak i_0 V\\
\mathfrak X'_s @>\pi_s >> \mathfrak X_0
\end{CD}\end{equation*}  
is also Cartesian, therefore $\mathfrak i_0^*\circ {\pi_s}_!={\pi'_s}_!\circ {\mathfrak i'_s}^*$. We can conclude that
\begin{align*}
\int_{\mathfrak U}|\omega|= {\pi'_s}_!\int_{\mathfrak U'}|\omega|={\pi'_s}_!\Big( {i'_s}^*\int_{\mathfrak X'}|\omega|\Big)
=(\mathfrak i_0^*\circ {\pi_s}_!)\int_{\mathfrak X'}|\omega|
=\mathfrak i_0^*\int_{\mathfrak X}|\omega|,
\end{align*}	
which hold in $\mathscr M_{\mathfrak X_0}^G$. Here, the second equality is due to Proposition \ref{dilatation_stft}. 
\end{proof}

\begin{corollary}[Additivity of motivic integrals]\label{int-additive}
Let $G$ be a smooth finite group $k$-scheme. Let $\mathfrak X$ be a generically smooth special formal $R$-scheme endowed with good adic action of $G$ and let $\omega$ be a gauge form on $\mathfrak X_\eta$. 
\begin{itemize}
\item[(i)] If $\{U_i, i \in I\}$ is a finite stratification of $\mathfrak X_0$ into $G$-invariant locally closed subsets, and $\mathfrak U_i$ is the formal completion of $\mathfrak X$ along $U_i$, then the following holds in $\mathscr M_{\mathfrak X_0}^G$:
$$\int_{\mathfrak X}|\omega|=\sum_{i\in I} (\varepsilon_i)_! \int_{\mathfrak U_i}|\omega|,$$
where $(\varepsilon_i)_!$ is the pushforward $\mathscr M_{U_i}^G\to  \mathscr M_{\mathfrak X_0}^G$ induced by the inclusion $\varepsilon_i: U_i\hookrightarrow \mathfrak X_0$.

\item[(ii)] If $\{\mathfrak U_i, i \in I\}$ is a finite covering of $\mathfrak X$ consisting of $G$-invariant open subsets, then the following holds in $\mathscr M_{\mathfrak X_0}^G$:
$$\int_{\mathfrak X}|\omega|=\sum_{I'\subseteq I}(-1)^{|I'|-1} (\varepsilon_{I'})_! \int_{\mathfrak U_{I'}}|\omega|,$$
where $\mathfrak U_{I'}=\cap_{i\in I'} \mathfrak U_i$ for all $I'\subseteq I$, and $(\varepsilon_{I'})_!$ is the pushforward $\mathscr M_{U_{I'}}^G\to  \mathscr M_{\mathfrak X_0}^G$ induced by the inclusion $\varepsilon_{I'}: U_{I'}\hookrightarrow \mathfrak X_0$.
\end{itemize}
\end{corollary}

\begin{proof}
This is an immediate consequence of Proposition \ref{base_change}. Indeed, by Proposition \ref{base_change}, $\int_{\mathfrak U_i}|\omega|={\varepsilon_i}^*\int_{\mathfrak X}|\omega|$, thus $(\varepsilon_i)_!\int_{\mathfrak U_i}|\omega|=(\varepsilon_i)_!({\varepsilon_i})^*\int_{\mathfrak X}|\omega|$, which implies the first statement by summing up the identities over $i\in I$. The second one is proved in the same way.	
\end{proof}

\begin{proposition}\label{special-connected}
Let $\mathfrak X$ be a smooth special formal $R$-scheme of pure relative dimension $d$ endowed with a good adic $G$-action. Suppose that $\omega$ is an $\mathfrak X$-bounded gauge form on $\mathfrak X_{\eta}$. Let $\mathcal C(\mathfrak X_0)$ be a family of constructible subsets of $\mathfrak X_0$ such that
\begin{itemize}
	\item $\mathfrak X_0=\bigcup_{C\in \mathcal C(\mathfrak X_0)}C$; $C\cap C'=\emptyset$ if $C, C'\in \mathcal C(\mathfrak X_0)$ and $C\not= C'$;
	
	\item every $C\in \mathcal C(\mathfrak X_0)$ is a union of connected components $C_1 ,\dots, C_r$ of $\mathfrak X_0$ such that $\ord_{C_1}(\omega)=\cdots=\ord_{C_r}(\omega)$ (denote this number by $\ord_C(\omega)$) and $C$ is $G$-invariant.
\end{itemize}
Then the following identity holds in $\mathscr M_{\mathfrak X_0}^G$:
$$\int_{\mathfrak X}|\omega|=\L^{-d}\sum_{C\in \mathcal C(\mathfrak X_0)}\left[ C\hookrightarrow \mathfrak X_0\right] \L^{-\ord_C(\omega)}.$$
\end{proposition}

\begin{proof}
This is an adic $G$-action analogue of the proof of \cite[Prop. 5.14]{Ni2}. We can assume that $\mathfrak X$ is flat. By Corollary \ref{int-additive} and \cite[Cor. 5.12]{Ni2}, we can also assume that $\mathfrak X_0$ is connected. Let $\pi\colon \mathfrak U\to \mathfrak X$ be a $G$-ditalation with center $\mathfrak X_0$. Then $\int_{\mathfrak X}|\omega|=(\pi_0)_!\int_{\mathfrak U}|\pi_{\eta}^*\omega|$. Since $\mathfrak U$ is a flat smooth stft formal $R$-scheme of pure relative dimension $d$, we deduce from the proof of \cite[Prop. 4.15]{Ni2} that $[\mathfrak U_0\to \mathfrak X_0]=[\mathfrak X_0\to \mathfrak X_0] \L^{\ord_{\varpi}\Jac_{\pi}}$. In particular, $\mathfrak U_0=\mathfrak U_s$ is connected, thus by Lemma \ref{stft-connected},
$$\int_{\mathfrak U}|\pi_{\eta}^*\omega|=\L^{-d}\left[\mathfrak U_0\to \mathfrak U_0\right] \L^{-\ord_{\mathfrak U_0}(\pi_{\eta}^*\omega)}$$
in $\mathscr M_{\mathfrak U_0}^G$. There is a small mistake in \cite[Lem. 5.13]{Ni2} concerning the order formula (compare it with \cite[Lem. 5.5]{Ni2}); and we can correct it as follows $\ord_{\mathfrak U_0}(\pi_{\eta}^*\omega)=\ord_{\mathfrak X_0}(\omega)+\ord_{\varpi}\Jac_{\pi}$. Then we have
$$\int_{\mathfrak X}|\omega|=\L^{-d}\left[\mathfrak X_0\to \mathfrak X_0\right] \L^{-\ord_{\mathfrak X_0}(\omega)},$$
which holds in $\mathscr M_{\mathfrak X_0}^G$. The proposition is now proved.
\end{proof}

\subsection{Monodromic volume Poincar\'e series and motivic volumes}\label{motvol}
A special formal $R$-scheme $\mathfrak X$ is called {\it regular} if $\mathcal O_{\mathfrak X,x}$ is regular for every $x\in \mathfrak X$. Assume that $d$ is the pure relative dimension of $\mathfrak X$. A closed formal subscheme $\mathfrak E$ of $\mathfrak X$ is called a {\it strict normal crossings divisor} if, for every $x$ in $\mathfrak X$, there exists a regular system of local parameters $(x_0,\dots,x_d)$ in $\mathcal O_{\mathfrak X,x}$ such that the ideal defining $\mathfrak E$ at $x$ is locally generated by $\prod_{i=0}^dx_i^{N_i}$ for some $N_i\in \mathbb N$, $0\leq i\leq d$, and such that the irreducible components of $\mathfrak E$ are regular (see \cite[Sect. 2.4]{Ni2} for definition of irreducibility). If $\mathfrak E'$ is an irreducible component of $\mathfrak E$ defined locally at $x$ by the ideal $(x_i^{N_i})$, it is a fact that $N_i$ is constant when $x$ varies on $\mathfrak E'$, which is called {\it the multiplicity of $\mathfrak E'$} and denoted by $N(\mathfrak E')$. 
Then we have $\mathfrak E=\sum_{i=1}^rN(\mathfrak E_i)\mathfrak E_i$, where $\mathfrak E_i$'s are irreducible components of $\mathfrak E$. The divisor $\mathfrak E$ is called a {\it tame strict normal crossings divisor} if $N(\mathfrak E_i)$ is prime to the characteristic exponent of $k$ for every $i$. Any special formal $R$-scheme $\mathfrak X$ is said to {\it have tame strict normal crossings} if $\mathfrak X$ is regular and $\mathfrak X_s$ is a tame strict normal crossings divisor.

\begin{definition}
Let $\mathfrak X$ be a flat generically smooth special formal $R$-scheme. A {\it resolution of singularities} of $\mathfrak X$ is a proper morphism of flat special formal $R$-schemes $\mathfrak h\colon\mathfrak Y\to\mathfrak X$ such that $\mathfrak h_{\eta}$ is an isomorphism and $\mathfrak Y$ is regular with $\mathfrak Y_s$ being a strict normal crossings divisor. The resolution of singularities $\mathfrak h$ is said to be {\it tame} if $\mathfrak Y_s$ is a tame strict normal crossings divisor.
\end{definition}

\begin{theorem}[Temkin \cite{Tem18}]\label{resolution-sing}
Suppose that the base field $k$ has characteristic zero. Then any generically smooth flat special formal $R$-scheme $\mathfrak X$ admits a resolution of singularities. 
\end{theorem} 

As shortly explained in \cite[Thm. 6.3.3]{CNS}, this theorem can be proved using \cite[Thm. 1.1.13]{Tem18}. For the affine case, it was realized early in \cite[Prop. 2.43]{Ni2} by means of a result in \cite{Tem}. 


\begin{notation}
For $n\in \mathbb N^*$, we define $R(n)=R[\tau]/(\tau^n-\varpi)$, $K(n)=K[\tau]/(\tau^n-\varpi)$, $\mathfrak X(n)=\mathfrak X\times_RR(n)$, and $\mathfrak X_{\eta}(n)=\mathfrak X_{\eta}\times_KK(n)$. If $\omega$ is a gauge form on $\mathfrak X_{\eta}$, let $\omega(n)$ be its pullback via $\mathfrak X_{\eta}(n)\to \mathfrak X_{\eta}$, a gauge form on $\mathfrak X_{\eta}(n)$.
\end{notation}

\begin{lemma}\label{mu_n-action}
Let $\mathfrak X$ be a formal $R$-scheme and $n$ in $\mathbb N^*$. Then there is a natural good adic $\mu_n$-action on both $\Spf R(n)$ and $\mathfrak X(n)$ which is induced from the ring homomorphism $R(n)\to k[\xi]/(\xi^n-1)\otimes_kR(n)$ given by $\tau\mapsto \xi\otimes \tau$. Moreover, the structural morphism of the formal $\Spf R(n)$-scheme $\mathfrak X(n)$ is $\mu_n$-equivariant.
\end{lemma}

Studying action of $\mu_n=\Spec\left(k[\xi]/(\xi^n-1)\right)$ on $\mathfrak X$, the previous lemma is straightforward. Remark that if $\mathfrak X$ is a generically smooth special formal $R$-scheme and $n\in \mathbb N^*$, then $\mathfrak X(n)$ is a generically smooth special formal $R(n)$-scheme.

\begin{definition}\label{Poincareseries}
Let $\mathfrak X$ be a generically smooth special formal $R$-scheme, and $\omega$ a gauge form on $\mathfrak X_{\eta}$. The below series is called the {\it motivic volume Poincar\'e series} of $\mathfrak X$:
$$
P(\mathfrak X,\omega; T):=\sum_{n\geq 1}\Big(\int_{\mathfrak X(n)}|\omega(n)|\Big)T^n\ \in \mathscr M_{\mathfrak X_0}^{\hat\mu}[[T]].
$$
\end{definition}

Let $\mathfrak X$ be a generically smooth flat special formal $R$-scheme of pure relative dimension $d$. Assume that $\mathfrak X$ admits a resolution of singularities $\mathfrak h\colon \mathfrak Y \to \mathfrak X$ (this is evident if the characteristic of $k$ is zero, cf. Theorem \ref{resolution-sing}). Let $\mathfrak E_i$, $i\in S$, be the irreducible components of $(\mathfrak Y_s)_{\mathrm{red}}$. Let $N_i$ be the multiplicity of $\mathfrak E_i$ in $\mathfrak Y_s$. Put $E_i:=(\mathfrak E_i)_0$ for $i\in S$ and
$$\mathfrak E_I:=\bigcap_{i\in I}\mathfrak E_i,\quad E_I:=\bigcap_{i\in I}E_i, \quad E_I^{\circ}:=E_I\setminus\bigcup_{j\not\in I}E_j$$ 
for any nonempty subset $I\subseteq S$. We can check that $\mathfrak E_I$ is regular and that $E_I=(\mathfrak E_I)_0$. Let $\{\mathfrak U=\Spf \mathcal O\}$ be a covering by affine open formal $R$-subschemes  of the formal completion $\mathfrak Y_I$ of $\mathfrak Y$ along $E_I^{\circ}$. The composition $\mathfrak f\circ\mathfrak h\colon \mathfrak U\to \Spf R$ is defined in the ring level by $\varpi\mapsto u\prod_{i\in I}y_i^{N_i}$, 
where $u$ is nonzero on $\mathfrak U$ and $y_i$ is a local coordinate defining $E_i$. Put $N_I:=\gcd(N_i)_{i\in I}$. Then we can construct as in \cite{NS} and \cite{Ni2} an unramified Galois covering $\widetilde{E}_I^{\circ}\to E_I^{\circ}$ with Galois group $\mu_{N_I}$ which is given over $\mathfrak U_0\cap E_I^{\circ}$ as the reduction $\widetilde{\mathfrak U}_0$ of the formal scheme 
$$\widetilde{\mathfrak U}=\Spf \mathcal O[T]/(uT^{N_I}-1).$$
Notice that $\widetilde{E}_I^{\circ}$ is endowed with natural good adic $\mu_{N_I}$-action over $E_I^{\circ}$ induced by multiplying  $T$ with elements of $\mu_{N_I}$. Let $\big[\widetilde E_I^{\circ}\big]$ be the class of the $\mu_{N_I}$-equivariant morphism $\widetilde E_I^{\circ}\to E_I^{\circ}\to \mathfrak X_0$ in the ring $\mathscr M_{\mathfrak X_0}^{\mu_{N_I}}$.

\begin{theorem}\label{int_Xm}
Let $\mathfrak X$ be a generically smooth flat special formal $R$-scheme of pure relative dimension $d$. Assume that we have a tame resolution of singularities $\mathfrak h\colon \mathfrak Y \to \mathfrak X$ with $\mathfrak Y_s=\sum_{i\in S}N_i\mathfrak E_i$ and an $\mathfrak X$-bounded gauge form $\omega$ on $\mathfrak X_{\eta}$ with order $\alpha_i:=\ord_{\mathfrak E_i}(\mathfrak h_{\eta}^*\omega)$ for $i\in S$. If $n\in \mathbb N^*$ is prime to the characteristic exponent of $k$, then the below identity holds in $\mathscr M_{\mathfrak X_0}^{\mu_n}$:
$$
\int_{\mathfrak X(n)}|\omega(n)|=\L^{-d}\sum_{\emptyset\not=I\subseteq S}(\L-1)^{|I|-1}\big[\widetilde{E}_I^{\circ}\big]\left(\sum_{\begin{smallmatrix} k_i\geq 1, \sum_{i\in I}k_iN_i=n \end{smallmatrix}} \L^{-\sum_{i\in I}k_i\alpha_i}\right).
$$
\end{theorem}

To prove this theorem we need the following two lemmas, the first one is trivial.

\begin{lemma}\label{lemma-preparation1}
Let $\mathfrak h\colon \mathfrak Y\to \mathfrak X$ be a resolution of singularities of a generically smooth special formal $R$-schemes. Then $\mathfrak h(n)\colon \mathfrak Y(n) \to\mathfrak X(n)$ is an adic $\mu_n$-equivariant morphism of formal $R(n)$-schemes with $\mathfrak h(n)_{\eta}$ an isomorphism.
\end{lemma}

Let $\mathfrak Y$ be a regular special formal $R$-scheme whose special fiber is a tame strict normal crossings divisor $\mathfrak Y_s=\sum_{i\in S}N_i\mathfrak E_i$. Recall from \cite[Def. 2.38]{Ni2} that a number $n\in \mathbb N^*$ is said to be {\it $\mathfrak Y_s$-linear} if there exists a nonempty subset $I\subseteq S$ of cardinal $|I|\geq 2$ such that $E_I^{\circ}\not=\emptyset$ and the linear equation $\sum_{i\in I}k_iN_i=n$ has solutions in $(\mathbb N^*)^I$.

\begin{lemma}\label{lemma-preparation2}
Let $\mathfrak Y$ be a regular special formal $R$-scheme whose special fiber is a tame strict normal crossings divisor $\mathfrak Y_s=\sum_{i\in S}N_i\mathfrak E_i$. If $n$ is prime to the characteristic exponent of $k$ and not $\mathfrak Y_s$-linear, then $\phi\colon Sm(\widetilde{\mathfrak Y(n)})\to \mathfrak Y(n)$ is an adic $\mu_n$-equivariant morphism of formal $R(n)$-schemes such that  $\phi_{\eta}$ is an open embedding and $Sm(\widetilde{\mathfrak Y(n)})_{\eta}(K(n)^{sh})=\mathfrak Y(n)_{\eta}(K(n)^{sh})$. Furthermore, 
$$Sm(\widetilde{\mathfrak Y(n)})_0=\bigsqcup_{N_i\mid n}\Big((\widetilde{\mathfrak Y(n)})_0\times_{\mathfrak Y_0}E_i^{\circ}\Big),$$
in which for $i$ with $N_i|n$, $(\widetilde{\mathfrak Y(n)})_0\times_{\mathfrak Y_0}E_i^{\circ}$ is a $\mu_{N_i}$-invariant constructible subset of $Sm(\widetilde{\mathfrak Y(n)})_0$ and $\mu_{N_i}$-equivariant canonically isomorphic to $\widetilde{E}_i^{\circ}$ over $E_i^{\circ}$.
\end{lemma}

\begin{proof}
By \cite[Thm. 5.1]{Ni2}, $\phi$ is a morphism of special formal $R(n)$-scheme which is the restriction of the normalization $\widetilde{\mathfrak Y(n)}\to \mathfrak Y(n)$. The normalization is clearly an adic $\mu_n$-equivariant, deduced from the natural adic $\mu$-action shown in Lemma \ref{mu_n-action}. Since $\mu_n$ is smooth, the second observation in the proof of Proposition \ref{G-smoothening} shows that $Sm(\widetilde{\mathfrak Y(n)})$ is $\mu_n$-invariant in $\widetilde{\mathfrak Y(n)}$, thus $\phi$ is an adic $\mu_n$-equivariant morphism of formal $R(n)$-schemes. The properties that $\phi_{\eta}$ is an open embedding, $Sm(\widetilde{\mathfrak Y(n)})_{\eta}(K(n)^{sh})=\mathfrak Y(n)_{\eta}(K(n)^{sh})$ and the decomposition of $Sm(\widetilde{\mathfrak Y(n)})_0$ in the lemma are verified in the proof of \cite[Thm. 5.1]{Ni2}.

Let $y$ be any point in $\mathfrak E_i^{\circ}$, and let $\mathfrak U=\Spf A$ be an affine open formal neighborhood of $y$ in $\mathfrak Y$.
Then the formal $R$-scheme structure at $y$ is given by $\varpi=uy_i^{N_i}$
with $u$ a unit. As in the proof of \cite[Lem. 4.4]{NS} we may write
$$\widetilde{\mathfrak U(n)}= \Spf(A\otimes_RR(n))\{T\}/(\varpi(n)^{n/N_i}T-y_i,uT^{N_i}-1),$$
where $\varpi(n)$ is the uniformizing parameter of $R(n)$. Then we have
$$(\widetilde{\mathfrak U(n)})_0\times_{\mathfrak U_0}E_i^{\circ}\cong \Spec\Big(A[T]/(y_i,uT^{N_i}-1)\Big),$$
it is endowed with the $\mu_{N_i}$-action $T\mapsto \xi\otimes T$ and is $\mu_{N_i}$-equivariant canonically isomorphic to the restriction of $\widetilde{E}_i^{\circ}$ over $\mathfrak U_0$. The conclusion comes from the glueing procedure. 
\end{proof}

\begin{proof}[Proof of Theorem \ref{int_Xm}]
We first prove the theorem for the case where $n$ is not $\mathfrak Y_s$-linear. Since the conclusion of Lemma \ref{lemma-preparation1} satisfies the hypothesis of Theorem \ref{changeofvariables}, we deduce that
\begin{align*}
\int_{\mathfrak X(n)}|\omega(n)|=(\mathfrak h(n)_0)_!\int_{\mathfrak Y(n)}|\mathfrak h(n)_{\eta}^*\omega(n)|.
\end{align*}
Similarly, applying Theorem \ref{changeofvariables} once again, in the setting of Lemma \ref{lemma-preparation2}, we get
\begin{align*}
\int_{\mathfrak Y(n)}|\mathfrak h(n)_{\eta}^*\omega(n)|=(\phi_0)_!\int_{Sm(\widetilde{\mathfrak Y(n)})}|\phi_{\eta}^*\mathfrak h(n)_{\eta}^*\omega(n)|.
\end{align*}
By Lemma \ref{lemma-preparation2}, all $D_i:=(\widetilde{\mathfrak Y(n)})_0\times_{\mathfrak Y_0}E_i^{\circ}$ for $N_i|n$ are $\mu_{N_i}$-invariant constructible subsets of $Sm(\widetilde{\mathfrak Y(n)})_0$. Using \cite[Prop. 7.11]{Ni2} and the notation in Proposition \ref{special-connected} we get
$$\ord_{D_i}(\phi_{\eta}^*\mathfrak h(n)_{\eta}^*\omega(n))=(n/N_i)\cdot\ord_{\mathfrak E_i}(\mathfrak h_{\eta}^*\omega)=(n/N_i)\cdot\alpha_i$$
for all $i\in S$ with $N_i|n$, we deduce from Proposition \ref{special-connected} and, again, Lemma \ref{lemma-preparation2} that
$$\int_{Sm(\widetilde{\mathfrak Y(n)})}|\phi_{\eta}^*\mathfrak h(n)_{\eta}^*\omega(n)|=\L^{-d}\sum_{N_i\mid n}\big[\widetilde{E}_i^{\circ} \big] \L^{-n\alpha_i/N_i}.$$
Therefore, the case where $n$ is not $\mathfrak Y_s$-linear has been completely proved.

For the case where $n$ is $\mathfrak Y_s$-linear, we can extend the computation in \cite[Lem. 7.5]{NS} to the special formal scheme setting and to the $\mu_n$-equivariant setting, which is natural, and use the same arguments in the proof of \cite[Thm. 7.12]{Ni2} (see also the proof of \cite[Thm. 7.6.]{NS}).
\end{proof}

\begin{corollary}\label{poincare}
Assume that the base field $k$ has characteristic zero. Let $\mathfrak X$ be a generically smooth flat special formal $R$-scheme of relative dimension $d$. Let $\mathfrak h\colon \mathfrak Y \to \mathfrak X$ be a resolution of singularities with $\mathfrak Y_s=\sum_{i\in S}N_i\mathfrak E_i$. Assume that $\omega$ is an $\mathfrak X$-bounded gauge form on $\mathfrak X_{\eta}$ with $\alpha_i:=\ord_{\mathfrak E_i}(\mathfrak h_{\eta}^*\omega)$ for $i\in S$. Then, in $\mathscr M_{\mathfrak X_0}^{\hat\mu}[[T]]$, 
$$P(\mathfrak X,\omega;T)=\L^{-d}\sum_{\emptyset\not=I\subseteq S}(\L-1)^{|I|-1}\big[\widetilde{E}_I^{\circ}\big]\prod_{i\in I}\frac{\L^{-\alpha_i}T^{N_i}} {1-\L^{-\alpha_i}T^{N_i}}.$$
\end{corollary}

By this corollary, the limit $-\lim_{T\to\infty}P(\mathfrak X,\omega;T)=\L^{-d}\sum_{\emptyset\not=I\subseteq S}(1-\L)^{|I|-1}\big[\widetilde{E}_I^{\circ}\big]$ in $\mathscr M_{\mathfrak X_0}^{\hat\mu}$ is independent of the choice of $\omega$. It depends on $\widehat{K^s}$ because it depends on $\varpi$ (see \cite[Rem. 7.40]{Ni2}), it is called the {\it motivic volume} of $\mathfrak X$ and denoted by $\MV(\mathfrak X;\widehat{K^s})$.

\begin{proposition}[Additivity of $\MV$]\label{MV-additive}
Suppose that $k$ has characteristic zero. Let $\mathfrak X$ be a generically smooth special formal $R$-scheme. The following hold.
\begin{itemize}
\item[(i)] If $\{U_i, i \in Q\}$ is a finite stratification of $\mathfrak X_0$ into locally closed subsets, and $\mathfrak U_i$ is the formal completion of $\mathfrak X$ along $U_i$, then $\MV(\mathfrak X;\widehat{K^s})=\sum_{i\in Q} \MV(\mathfrak U_i;\widehat{K^s})$.

\item[(ii)] If $\{\mathfrak U_i, i \in Q\}$ is a finite open covering of $\mathfrak X$, then by putting $\mathfrak U_{I}=\bigcap_{i\in I} \mathfrak U_i$, we have
$$\MV(\mathfrak X;\widehat{K^s})=\sum_{\emptyset\not=I\subseteq Q}(-1)^{|I|-1} \MV(\mathfrak U_{I};\widehat{K^s}).$$
\end{itemize}
\end{proposition}

\begin{proof}
Since $\mathfrak X$ admits a resolution of singularities (cf. Theorem \ref{resolution-sing}), we can identify $\mathfrak X$ with its resolution of singularities. It implies from \cite[Prop.-Def. 7.38]{Ni2} that $\mathfrak X$ has a finite open covering $\{\mathfrak V_j\}_j$ such that each $\mathfrak V_j$ admits a $\mathfrak V_j$-bounded gauge form $\omega$ on $\mathfrak V_{j\eta}$. Thus we can apply Corollary \ref{int-additive} to the coefficients of $P(\mathfrak U_i\cap \mathfrak V_j,\omega;T)$ and deduce the proposition.	
\end{proof}

\subsection{Motivic zeta functions and motivic nearby cycles of formal power series}\label{Mot_series}
Consider the mixed formal power series $R$-algebra $R\{x\}[[y]]$, with $x=(x_1,\dots, x_m)$ and $y=(y_1,\dots,y_{m'})$. Let $d=m+m'$. Let $f$ be in $k\{x\}[[y]]$ and that $f(x,0)$ is non-constant. Let $\mathfrak X_f$ be the formal completion of $\Spf(k\{x\}[[y]])$ along $(f)$. It is a special formal $R$-scheme of pure relative dimension $d-1$, with structural morphism defined by $\varpi \mapsto f$. The reduction of $\mathfrak X_f$ is the algebraic $k$-variety $(\mathfrak X_f)_0=\Spec k[x]/(f(x,0))$. 

\begin{lemma}\label{modeloff}
Let $f$ be in $k\{x\}[[y]]$ such that $f(x,0)$ is non-constant. Then, there is an isomorphism of special formal $R$-schemes $\mathfrak X_f\cong \Spf\left( R\{x\}[[y]]/(f-\varpi)\right)$.
Consequently, $\mathfrak X_f$ is a generically smooth special formal $R$-scheme of pure relative dimension $d-1$.
\end{lemma}

\begin{proof}
Consider the $k$-algebra homomorphism $\varphi\colon R\{x\}[[y]] \to \widehat{R\{x\}[[y]]_{(f)}}$ given by $\varphi(\varpi)=f$, $\varphi(x)=x$ and $\varphi(y)=y$. This induces an isomorphism $\overline{\varphi}\colon R\{x\}[[y]]/(f-\varpi) \to \widehat{k\{x\}[[y]]_{(f)}}$, which makes the following diagram
\begin{displaymath}
\xymatrix{
	R\{x\}[[y]]/(f-\varpi)\ \ar[rr]^{\ \quad \overline{\varphi}}\ar@{<-}[dr]_{\varpi\ \! \mapsto\ \! [\varpi]}&&\ \widehat{k\{x\}[[y]]_{(f)}}\ar@{<-}[dl]^{\varpi\ \! \mapsto\ \! f}\\
	&\ R\ &
}
\end{displaymath}
commutes. The lemma is now proved.
\end{proof}

By \cite[Sect. 4]{AJP}, we can see that $\mathfrak X_f$ is a formal scheme of pseudo-finite type over $k$, the sheaf of continuous differential form $\Omega^i_{\mathfrak X_f/k}$ is coherent for any $i$, and that there exists a morphism of coherent $\mathcal O_{\mathfrak X_f}$-modules $d\varpi \wedge(.) \colon \Omega_{\mathfrak X_f/R}^{d-1}\to \Omega_{\mathfrak X_f/k}^d$ defined by taking the exterior product with the differential $df$. By \cite[Sect. 7]{deJ} and \cite[Lem. 2.5]{Ni2}, we have an exact functor $(\cdot)_{\rig}$ from the category of coherent $\mathcal O_{\mathfrak X_f}$-modules to the category of coherent $\mathcal O_{(\mathfrak X_f)_{\eta}}$-modules. Taking this functor we get a morphism of coherent $\mathcal O_{(\mathfrak X_f)_{\eta}}$-modules 
$$d\varpi \wedge(\cdot) \colon \Omega_{(\mathfrak X_f)_{\eta}/K}^{d-1}\to (\Omega_{\mathfrak X_f/k}^d)_{\rig}.$$
By \cite[Prop. 7.19]{Ni2}, the morphism $d\varpi \wedge(\cdot)$ is an isomorphism. If $\omega$ is a global section of $\Omega_{\mathfrak X_f/k}^d$, we denote as in \cite[Def. 7.21]{Ni2} by $\omega/df$ the inverse image of $\omega$ under $d\varpi \wedge(\cdot)$. We fix a gauge form $\omega$ on $\mathfrak X_f$. 
Let $\mathfrak h\colon \mathfrak Y\to \mathfrak X_f$ be a tame resolution of singularities of $\mathfrak X_f$. Then, by \cite[Lem. 7.24]{Ni2}, $(\mathfrak h^*\omega)/d\varpi=\mathfrak h_{\eta}^*(\omega/df)$ in $\Omega_{\mathfrak Y_{\eta}/K}^{d-1}(\mathfrak Y_{\eta})$. Since $\mathfrak h$ is a tame resolution of singularities of $\mathfrak X_f$ and $\omega$ is a gauge form on $\mathfrak X_f$, $\mathfrak h^*\omega$ is a gauge form on $\mathfrak Y$. Since $\mathfrak Y$ is a regular flat special formal $R$-scheme, it follows from \cite[Cor. 7.23]{Ni2} that $(\mathfrak h^*\omega)/dt$ is a $\mathfrak Y$-bounded gauge form on $\mathfrak Y_{\eta}$. This together with Remark \ref{remark2} guarantees that the integral $\int_{\mathfrak X_f(n)}|(\omega/df)(n)|$ makes sense as a motivic quantity in $\mathscr M_{(\mathfrak X_f)_0}^{\mu_n}$ even though $\omega/df$ is possibly not a gauge form.

Assume that the data of $\mathfrak Y$ are given as in the setting before Theorem \ref{int_Xm} and that $K_{\mathfrak Y/\mathfrak X_f}=\sum_{i\in S}(\nu_i-1)\mathfrak E_i$. Using the same argument in the proof of \cite[Lem. 7.30]{Ni2} we get $\ord_{\mathfrak E_i}\mathfrak h_{\eta}^*(\omega/df)=\nu_i-N_i$ for all $i\in S$. Note that these numbers do not depend on $\omega$. Similarly as in the proof of Theorem \ref{int_Xm} we have the following result.

\begin{proposition}\label{mzf-fseries}
With the previous notation and hypotheses, if $n\in \mathbb N^*$ is prime to the characteristic exponent of $k$, the below identity holds in $\mathscr M_{(\mathfrak X_f)_0}^{\mu_n}$:
$$
\int_{\mathfrak X_f(n)}|(\omega/df)(n)|=\L^{n+1-d}\sum_{\emptyset\not=I\subseteq S}(\L-1)^{|I|-1}\big[\widetilde{E}_I^{\circ}\big]\left(\sum_{\begin{smallmatrix} k_i\geq 1, \sum_{i\in I}k_iN_i=n \end{smallmatrix}}\L^{\sum_{i\in I}k_i(N_i-\nu_i)}\right).
$$
If, in addition, $k$ has characteristic zero, then 
$$P(\mathfrak X_f,\omega/df;T)=\L^{-(d-1)} \frac{\L T}{1-\L T}\ast\sum_{\emptyset\not=I\subseteq S}(\L-1)^{|I|-1}\big[\widetilde{E}_I^{\circ}\big]\prod_{i\in I}\frac{\L^{-\nu_i}T^{N_i}} {1-\L^{-\nu_i}T^{N_i}},$$
where $\ast$ is the Hadamard product of formal series in $\mathscr M_{(\mathfrak X_f)_0}^{\hat\mu}[[T]]$ (cf. Section \ref{Hadamard}). Moreover, 
$$\MV(\mathfrak X_f;\widehat{K^s})=\sum_{\emptyset\not=I\subseteq S}(1-\L)^{|I|-1}\big[\widetilde{E}_I^{\circ}\big] \in \mathscr M_{(\mathfrak X_f)_0}^{\hat\mu}.$$
\end{proposition}

Consider $d=m$, $f\in k[x]$, $R=k[[t]]$ and $K=k((t))$, with $k$ of characteristic zero and $t$ replacing $\varpi$. Assume $X_0=f^{-1}(0)\not=\emptyset$. Let $\mathscr L_n(\mathbb A_k^d)$ be the $n$-jet scheme of $\mathbb A_k^d$ (cf. \cite{DL1, DL2}). The contact loci and motivic zeta function of $f$ are defined as follows
\begin{align*}
\mathscr{X}_n(f)&=\big\{\gamma\in \mathscr{L}_n(\mathbb A_k^d) \mid f(\gamma)= t^n  \!\! \mod t^{n+1}\big\},\\
Z_f(T)&=\sum_{n\geq 1}\big[\mathscr{X}_n(f)\big]\L^{-nd}T^n\ \in \mathscr M_{X_0}^{\hat{\mu}}[[T]],
\end{align*}
where the $\mu_n$-action on the $\mathscr{X}_n(f)$ is given by $\xi\cdot\gamma(t)=\gamma(\xi t)$. By \cite{DL1}, $Z_f(T)$ is rational and we have the {\it motivic nearby cycles} $\mathscr S_f=-\lim_{T\to\infty}Z_f(T)$ of $f$. For a closed point $\x\in X_0$, one also consider the local version $\mathscr{X}_{n,\x}(f)$ and $Z_{f,\x}(T)\in\mathscr M_k^{\hat{\mu}}[[T]]$ (cf. \cite{DL1}). 

\begin{corollary}\label{comparezeta}
With the previous notation and hypotheses, we have
\begin{align*}
[\mathscr X_n(f)]&=\L^{(n+1)(d-1)}\int_{\mathfrak X_f(n)}|(\omega/df)(n)|,\\
[\mathscr X_{n,\x}(f)]&=\L^{(n+1)(d-1)}\int_{\widehat{(\mathfrak X_f)_{/\x}}(n)}|(\omega/df)(n)|.
\end{align*}
Consequently, $Z_f(T)=U(T) \ast \L^{d-1} P(\mathfrak X_f,\omega/df;T)$, $Z_{f,\x}(T)=U(T)\ast \L^{d-1} P(\widehat{(\mathfrak X_f)_{/\x}},\omega/df;T)$, where $U(T)=\sum_{n\geq 1}(\L^{-1}T)^n$, from which $\mathscr S_f=\MV(\mathfrak X_f;\widehat{K^s})$ and $\mathscr S_{f,\x}=\MV(\widehat{(\mathfrak X_f)_{/\x}};\widehat{K^s})$.
\end{corollary}

\begin{proof}
Let $h\colon Y\to \mathbb A_k^d$ be an embedded resolution of singularities of $X_0$ with strict normal crossing divisor $Y_s=\sum_{i\in S}N_iE_i$. By \cite[Lem. 2.4]{NS}, the induced morphism $\widehat h\colon \widehat Y\to \mathfrak X_f$ is a resolution of singularities, where $\widehat Y$ is the formal completion of $Y$ along $(f\circ h)$ and $(\widehat Y)_s=Y_s$. Assume that $K_{Y/\mathbb A_k^d}=\sum_{i\in S}(\nu_i-1)E_i$. By \cite[Lem. 9.6]{NS}, $\ord_{E_i}\widehat h^*(\omega/df)=\nu_i-N_i$. By Remark \ref{remark2}, Theorem \ref{int_Xm} and \cite[Thm. 2.4]{DL5} we complete the proof.
\end{proof}

\begin{definition}\label{def-formal}
Let $k$ be a field of characteristic zero. Let $f$ be in $k\{x\}[[y]]$ such that $f(x,0)$ is non-constant. Let $\x$ be a closed point in $(\mathfrak X_f)_0$. The {\it motivc zeta function} of $f$ and the {\it local motivic zeta function} of $f$ at $\x$ are defined as $Z_f(T)=\L^{d-1}P(\mathfrak X_f,\omega/df;T)$ and $Z_{f,\x}(T)=\L^{d-1}P(\widehat{(\mathfrak X_f)_{/\x}},\omega/df;T)$. 
The {\it motivic nearby cycles } of $f$ and the {\it motivic Milnor fiber} of $f$ at $\x$ are defined as $\mathscr S_f=\L^{d-1}\MV(\mathfrak X_f;\widehat{K^s})$ and $\mathscr S_{f,\x}=\L^{d-1}\MV(\widehat{(\mathfrak X_f)_{/\x}};\widehat{K^s})$. 
\end{definition}


\subsection{Two conjectures}\label{last-ss}
Let $k$ be a field of characteristic zero, and let $f$ be in $k[[x_1,\dots,x_d]]$ such that $f(0)=0$. Let $R=k[[t]]$ and $K=k((t))$. Consider the special formal $R$-scheme $\mathfrak X_f=\Spf(R[[x_1,\dots,x_d]]/(f-t))$ of relative dimension $d-1$ (cf. Lemma \ref{modeloff}). Its reduction is $(\mathfrak X_f)_0=\Spec k$. As in Definition \ref{def-formal}, using a Gelfand-Leray form, we have the concept of motivic zeta function $Z_f(T)=\L^{d-1}P(\mathfrak X_f,\omega/df;T)\in \mathscr M_k^{\hat\mu}[[T]]$ and that of motivic Milnor fiber $\mathscr S_f=\L^{d-1}\MV(\mathfrak X_f;\widehat{K^s})\in \mathscr M_k^{\hat\mu}$ of the formal power series $f$. For $n\in \mathbb N^*$, consider the algebraic $k$-variety
$$\mathscr X_{n,0}(f)=\big\{\gamma\in \mathscr L_n(\mathbb A_k^d) \mid f_n(\gamma)=t^n\!\! \mod t^{n+1},\ \gamma(0)=0\big\},$$ 
which admits the good $\mu_n$-action given by $\xi\cdot \gamma(t)=\gamma(\xi t)$. Here, $f_n$ denotes the sum of all the degree $k$ homogeneous parts of $f$ over $1\leq k\leq n$. Note that, although $f$ is a formal power series, $\mathscr X_{n,0}(f)$ is really an algebraic $k$-variety.  


\begin{conjecture}\label{conj1}
Let $f$ be a formal power series in $k[[x_1,\dots,x_d]]$ such that $f(0)=0$. Put $\omega=dx_1\wedge \cdots\wedge dx_d$. Then the following identity holds in $\mathscr M_k^{\mu_n}$:
\begin{align*}
[\mathscr X_{n,0}(f)]=\L^{(n+1)(d-1)}\int_{\mathfrak X_f(n)}|(\omega/df)(n)|.
\end{align*}
\end{conjecture}

We consider the case $k=\mathbb{C}$. Let $f\in \mathbb C\{x_1,\dots,x_d\}$ be a complex analytic function vanishing at $O\in \mathbb C^d$. Using Denef-Loeser's theory of motivic integration \cite{DL1, DL2}, it seems impossible to define {\it directly} $\mathscr X_{n,O}(f)$ except $f$ is a polynomial, but we can define $\mathscr X_{n,O}(f)$ to be $\mathscr X_{n,O}(f_n)$ with $f_n$ understood as above. On the other hand, the rationality of the series $Z(T)=\sum_{n\geq 1}[\mathscr X_{n,O}(f_n)]\L^{-nd}T^n$ in $\mathscr M_{\mathbb C}^{\hat\mu}[[T]]$ is a big problem because of the lack of existence of a common log resolution for all hypersurfaces defined by the vanishing of $f_n$. 
Hence, if Conjecture \ref{conj1} is not proved yet, we can not define the motivic Milnor fiber of a complex analytic function $f$ as $-\lim_{T\to \infty}Z(T)$. We define {\it the motivic Milnor fiber} $\mathscr S_{f,O}$ of the complex analytic function germ $(f,O)$ to be the motivic Milnor fiber of a Taylor expansion of $f$ at $O$ as in Definition \ref{def-formal}. 
 Consider the topological Milnor fiber $F_{f,O}$ of $f$ at the origin. We use the same symbol $\chi_c$ for the topological Euler characteristic with compect support and for the Euler characteristic with compact support of complex constructible sets.


\begin{conjecture}\label{conj2}
Let $f$ be a complex analytic function in $d$ variables which vanishes at $O$. Then the following equality holds.
$$\chi_c(\mathscr S_{f,O})=\chi_c(F_{f,O}).$$ 
\end{conjecture}





\end{document}